\documentclass[12pt,superscriptaddress]{amsart} 
\usepackage[text={6.5in,9in},centering]{geometry}

 \usepackage[dvips]{graphics}
  \usepackage{graphicx}       
 \usepackage{subfigure}

 \usepackage{epsfig}
\usepackage{color}
\usepackage{amssymb,amsmath}
\usepackage{amscd}
\usepackage{graphicx}
\usepackage[config,labelfont={rm,rm},textfont=rm,aboveskip=0in,belowskip=2ex]{caption,subfig}
\DeclareMathAlphabet\mathbold{OML}{cmm}{b}{it}
\graphicspath{{FigEPS/}{FigPNG/}{FigJPG/}{FigPDF/}}

\newtheorem{theorem}{Theorem}[section]

\newtheorem{lemma}[theorem]{Lemma}
\newtheorem{remark}[theorem]{Remark}
\newtheorem{assumption}[theorem]{Assumption}
\usepackage{bm}
\usepackage{xy}
\numberwithin{equation}{section}
\numberwithin{table}{section}
\numberwithin{figure}{section}

\newcommand{\bfM}{\mathbf{M}}
\newcommand{\bfQ}{\mathbf{Q}}

\newcommand{\bftau}{\boldsymbol{\tau}}
\newcommand{\bfgamma}{\boldsymbol{\gamma}}
\newcommand{\bfbeta}{\boldsymbol{\beta}}
\newcommand{\bfzeta}{\boldsymbol{\zeta}}
\newcommand{\bfeta}{\boldsymbol{\eta}}
\newcommand{\bfpsi}{\boldsymbol{\psi}}
\newcommand{\bfI}{\mathbf{I}}
\newcommand{\bfx}{\mathbf{x}}
\newcommand{\bfdiv}{\mathbf{div}}
\newcommand{\Div}{\text{div}}
\newcommand{\matrixD}{\mathcal{D}}
\newcommand{\bfepsilon}{\boldsymbol{\epsilon}}
\newcommand{\bfR}{\boldsymbol{R}}
\newcommand{\bft}{\boldsymbol{t}}
\newcommand{\bfn}{\boldsymbol{n}}
\newcommand{\grad}{\text{grad~}}
\newcommand{\curl}{\text{curl~}}
\newcommand{\curlh}{\text{curl}_h}
\newcommand{\rot}{\text{rot}}
 \author[Guozhu Yu]{Guozhu Yu$^1$}
 \author[Xiaoping Xie]{Xiaoping Xie$^{2,*}$} \thanks{*: Corresponding author.}

 \author[Yuanhui Guo]{Yuanhui Guo$^3$}

 \title[Analysis of  hybrid  methods of mixed-shear-projected 
elements]
 {Analysis of hybrid  methods of mixed-shear-projected triangular and quadrilateral 
elements for Reissner-Mindlin plates}
 \dedicatory{$^1$School of Mathematics, Southwest Jiaotong University, Chengdu  610031, China\\
             $^2$School of Mathematics, Sichuan University, Chengdu  610064, China\\
             $^3$Experiment Center, China West Normal University, Nanchong, Sichuan 637009, China}

 \thanks{Email addresses: yuguozhumail@gmail.com (G. Yu), xpxie@scu.edu.cn (X. Xie), gyh6209@sina.com (Y. Guo)}

\begin{document}
\setlength{\abovecaptionskip}{0pt}
\setlength{\belowcaptionskip}{0pt}
\begin{abstract}
It is known that the 3-node hybrid triangular element MiSP3 and 4-node hybrid quadrilateral element MiSP4 presented by Ayad, Dhatt  and Batoz
(Int. J. Numer. Meth. Engng 1998, 42: 1149-1179) for Reissner-Mindlin plates behave robustly in
numerical benchmark tests.  These two elements are based on Hellinger-Reissner variational principle, where  continuous piecewise linear/isoparametric bilinear  interpolations, as well as the mixed shear interpolation/projection technique of MITC family,  are used for  the approximations of displacements, and piecewise-independent equilibrium modes are used for the approximation of bending moments/shear stresses.  We show that the MiSP3 and MiSP4
elements are uniformly stable with respect to the plate thickness and thus free from shear-locking.


\end{abstract}


\maketitle

\section{Introduction}
Due to avoidance of $C^1$-continuity difficulty, the
Reissner-Mindlin (R-M) plate model is today the dominating
two-dimensional model used to calculate the bending of a thick/thin
three-dimensional plate of thickness $t$.
 It's well-known that for values of $t$ close to zero,
the standard low-order finite element discretization of this model
suffers from shear locking (\cite{Arnold.D1981,Hughes.T1987}).

To overcome the shear locking difficulty and derive `locking-free'
or robust plate bending elements that are valid for the analysis of
thick and thin plates, significant efforts are devoted to the
development of simple and efficient triangular and quadrilateral
finite elements 
in the past few
decades. The most common approach is to modify the variational
formulation with some reduction operator so as to weaken the
Kirchhoff constraint (see \cite{Arnold.D;Falk.R1989, Ayad.R;Dhatt.G;Batoz.J1998,
bathe1989mitc7, Bathe.K;Dvorkin.E1985,  batoz1980study, Batoz.J;Bentahar.M1982,
Boffi2008,Braess.D2001, Brezzi.F;Bathe.KJ;Fortin.M1989,    Brezzi.F;Fortin.M1991, Brezzi.F;Fortin.M;Stenberg.R1991,  DURAN.R;LIBERMAN.E1992, FALK.R;TU.T2000, Hu.J;Ming;Shi2003, Hu.J;Shi.Z2007, Hu.J;Shi.Z2008, Hu.J;Shi.Z2009, Hughes;1978,hughes1981linear, Hughes;1977, hughes1981finite,  LOVADINA.C2005, macneal1982derivation,Malkus;1978, MING.P;SHI.Z2001,
MING.P;SHI.Z2005, Ming.P;Shi.Z2006,
papadopoulos1990triangular,Pitkaranta.J;Suri.M1996,Zhang.Z;Zhang.S1994,
 zienkiewicz1990plate,Zienkiewicz;1971} and the references therein).

Among the existing elements, the family of finite elements named mixed interpolated tensorial components (MITC)  by Bathe et. al \cite{bathe1989mitc7,Bathe.K;Dvorkin.E1985}
is one of the most attractive representative.
By virtue of an independent shear
approximation and a discrete Mindlin technique along edges, MITC elements define the shear strains
in terms of the edge tangential strains that are projected on the element degrees of freedom.
As the lowest order quadrilateral MITC element, the 4-node plate element MITC4 is very likely the most used
in practice. Unfortunately, there is no so called low order triangular `MITC3' element. In other words, the 3-node
plate element MITC3 defined with the same technique of shear interpolation
produces very unsatisfactory
results, and, in general, it needs some kind of stabilization \cite{Brezzi.F;Fortin.M;Stenberg.R1991}.

 With the same technique of shear interpolation as in the
 element MITC family, Ayad, Dhatt and Batoz \cite{Ayad.R;Dhatt.G;Batoz.J1998} presented an
improved formulation for obtaining locking-free triangular and quadrilateral
elements, which are called MiSP3 and MiSP4 elements respectively. It is based on
Hellinger-Reissner variational principle,  including variables
of displacements, shear stresses and bending moments. For MiSP3 element continuous
piecewise linear interpolation is used for the approximations of displacements,
and a  piecewise-independent equilibrium mode is used for the approximation of bending moments/shear stresses.
While for MiSP4 element it adopts continuous
isoparametric bilinear displacement interpolation. The numerical experiments in \cite{Ayad.R;Dhatt.G;Batoz.J1998}
 showed that the MiSP3 and MiSP4 elements  both avoid locking
phenomenon. However, so far there is no uniform stability analysis for them
with respect to plate thickness.



The main goal of this work is to establish uniform convergence for
triangular MiSP3 element and quadrilateral MiSP4 element.
%
The key to the analysis of  MiSP3 is the
discrete Helmholtz decomposition in Lemma \ref{lemma: DiscreteDecomposition}, while for MiSP4
we use the property of the shear interpolation (Lemma \ref{lemma projection}) proved in \cite{Duran-Hernandez-Nieto-Liberman-Rodriguez2004}.

We arrange the rest of this paper as follows. In Section 2 we give
weak formulations of the model. Section 3 introduces the finite element spaces for MiSP3 and MiSP4 elements.
 We derive in Sections 4-5  uniform error
estimates for MiSP3 and MiSP4 elements, respectively. Finally in
Section 6 we provide some numerical results to verify the
theoretical results.

For convenience, throughout the paper we use the notation $a\lesssim
b$ to represent that there exists a generic positive constant C,
independent of the mesh parameter $h$ and the plate thickness $t$,
such that $a \le Cb$. We also abbreviate $a \lesssim b\lesssim$ as
$a \thickapprox  b$.

We will also use various standard differential operators:
$$\grad r=(\frac{\partial r}{\partial x},\frac{\partial r}{\partial y})^T,\
  \curl p=(\frac{\partial p}{\partial y},-\frac{\partial p}{\partial x})^T,\
  \Div \bfpsi=\frac{\partial \bfpsi_1}{\partial x}+\frac{\partial \bfpsi_2}{\partial y},\
  \rot \bfpsi=\frac{\partial \bfpsi_1}{\partial y}-\frac{\partial \bfpsi_2}{\partial x}.$$

\section{weak problem}
The Reissner-Mindlin model for the bending of a clamped isotropic
elastic plate in equilibrium reads as: Find $(w,\bfbeta)\in
H_0^1(\Omega)\times H_0^1(\Omega)^2$ such that
\begin{eqnarray}
- \bfdiv \matrixD\bfepsilon(\bfbeta) -\lambda t^{-2}(\grad w-\bfbeta)=0 &\mbox{in} &\Omega,\label{eq:WeakForm1-a}\\
 -\lambda t^{-2}\Div(\grad w-\bfbeta)=g  &\mbox{in} &\Omega.\label{eq:WeakForm1-b}
\end{eqnarray}
Here $\Omega\subset \mathbb{R}^{2}$, assumed to be a convex polygon for
simplicity, is the region occupied by the midsection of the plate
with plate thickness $t$, $w$ and $\bfbeta$ denote respectively the
transverse displacement of the midplane and the rotation of the
fibers normal to it, $\bfepsilon(\bfbeta)$ is the symmetric part of
the gradient of $\bfbeta$, $g$ is the transverse loading, $\matrixD$
is the elastic module tensor defined by
$$\matrixD\bfQ= \frac{E}{12(1-\nu^2)}[(1-\nu)\bfQ+\nu\mbox{tr}(\bfQ)\bfI]$$
with $\bfQ$ a $2\times 2$ symmetric matrix, $\lambda=\frac{\kappa
E}{2(1+\nu)}$  with $E$ the Young's modulus,  $\nu$ the Poisson's
ratio, and $\kappa=\frac{5}{6}$ the shear correction factor.

Set
$$\mathbb{M}:= L^2(\Omega )^{2\times 2}_{sym},\quad
\Gamma:=L^2(\Omega)^2, \quad W:=H_0^1(\Omega), \quad \Theta
:=H_0^1(\Omega)^2.$$
 When introducing  the shear stress vector
$\bfgamma=\lambda t^{-2}(\grad w-\bfbeta)$ and the bending moment
tensor $\bfM =-\matrixD{\bfepsilon(\bfbeta)}$, the model problem
(\ref{eq:WeakForm1-a})-(\ref{eq:WeakForm1-b}) changes into the
following system: Find $(\bfM,\bfgamma,w,\bfbeta)\in\mathbb{M}\times
\Gamma \times W\times \Theta$ such that
\begin{eqnarray}\label{eq:StrongForm2}
 {\bfdiv}\bfM -\bfgamma=0 &\mbox{in} &\Omega,\label{1st-Euler}\\
 \Div\bfgamma+g=0  &\mbox{in} &\Omega,\\
 \bfM+\matrixD\bfepsilon(\bfbeta)=0 &\mbox{in} &\Omega,\\
 \bfgamma-\lambda t^{-2}(\grad w-\bfbeta)=0 &\mbox{in}
 &\Omega.
\end{eqnarray}
The variational formulation of this system reads: Find
$(\bfM,\bfgamma,w,\bfbeta)\in\mathbb{M}\times \Gamma \times W\times
\Theta$ such that
\begin{eqnarray}
a(\bfM,\bfgamma;\bfQ,\bftau)+b(\bfQ,\bftau;w,\bfbeta)=0 \ \quad for\
all\ (\bfQ,\bftau)\in \mathbb{M}\times\Gamma,\label{eq:WeakForm2-a}
\\
b(\bfM,\bfgamma;v,\bfzeta)=-\int_{\Omega}gvd\bfx\ \quad for\
all\ (v,\bfzeta)\in W\times\Theta, \label{eq:WeakForm2-b}
\end{eqnarray}
where
 the bilinear forms
\begin{eqnarray}
\nonumber a(\cdot,\cdot;\cdot,\cdot):&(L^2(\Omega)^{2\times
2}_{sym}\times L^2(\Omega)^2)\times (L^2(\Omega)^{2\times
2}_{sym}\times L^2(\Omega)^2) &\to \mathbb{R},
\\\nonumber
b(\cdot,\cdot;\cdot,\cdot):&(L^2(\Omega)^{2\times 2}_{sym}\times
L^2(\Omega)^2)\times (H_0^1(\Omega) \times H_0^1(\Omega)^2)
 &\to \mathbb{R}
 \end{eqnarray}
are defined by
\begin{eqnarray}
a(\bfM,\bfgamma;\bfQ,\bftau):=\int_{\Omega}\bfM:\matrixD^{-1}\bfQ
d\bfx+\frac{t^2}{\lambda}\int_{\Omega}\bfgamma \cdot \bftau
d\bfx,
\\ b(\bfQ,\bftau;v,\bfzeta):=\int_{\Omega}\bfQ:\bfepsilon(\bfzeta)d\bfx-\int_{\Omega}\bftau \cdot (\grad
v-\bfzeta)d\bfx.
 \end{eqnarray}

In the latter analysis we will use the Helmholtz theorem: for any $\bftau\in L^2(\Omega)^2$,
\begin{equation}\label{eq:Helmholtz}
\bftau=\grad s+\curl q ,\mbox{~with~}(s,q)\in H_0^1(\Omega)\times\hat{H}^1(\Omega),
\end{equation}
where $$\hat{H}^1(\Omega):=\{ q\in {H}^1(\Omega): \int_\Omega q d\bfx=0\}.$$
Then the shear strain vector $\bfgamma$ can be decomposed as
\begin{equation}\label{eq:Helmholtz-gamma}
\bfgamma=\grad r+\curl p
\end{equation}
with $(r,p)\in H_0^1(\Omega)\times\hat{H}^1(\Omega)$. Moreover, since $\bfgamma\cdot\bft=0$ on
$\partial \Omega$, the decomposition \eqref{eq:Helmholtz-gamma} indicates that $p$ satisfies
$$\grad p\cdot \bfn=0 \quad \partial \Omega,$$
where $\bft$, $\bfn$ are respectively  the unit tangent vector and unit outer normal vector  along $\partial \Omega$. Then the model problem (\ref{eq:WeakForm1-a})-(\ref{eq:WeakForm1-b})
is also equivalent to the following system:

Find $(r,\bfbeta,p,w)\in H_0^1(\Omega)\times H_0^1(\Omega)^2\times \hat{H}^1(\Omega)\times H_0^1(\Omega)$ such that
\begin{eqnarray}
 &(\grad r, \grad v)=(g,v),\quad \forall v\in H_0^1(\Omega),\label{eq:WeakForm3-a}\\
 &(\bfepsilon(\bfbeta),\matrixD\bfepsilon(\bfzeta))-(\curl p,\bfzeta)=(\grad r,\bfzeta),\quad  \forall \bfzeta\in H_0^1(\Omega)^2, \label{eq:WeakForm3-b}\\
 &-(\bfbeta,\curl q)-\frac{t^{2}}{\lambda}(\curl p,\curl q)=0, \quad\forall q\in \hat{H}^1(\Omega),\label{eq:WeakForm3-c}\\
 &(\grad w,\grad s)=(\bfbeta+\frac{t^{2}}{\lambda}\grad r,\grad s),\quad \forall s\in H_0^1(\Omega). \label{eq:WeakForm3-d}
\end{eqnarray}

The following regularity results were proved by Arnold and Falk \cite{Arnold.D;Falk.R1989}.
\begin{theorem}
Let $\Omega$ be a convex polygon or smoothly bounded domain in the plane. For any $t\in (0,1]$ and any $g\in L^2(\Omega)$,
there exists a unique quadruple $(r,\bfbeta,p,w)\in H_0^1(\Omega)\times H_0^1(\Omega)^2\times \hat{H}^1(\Omega)\times H_0^1(\Omega)$
solving problem (\ref{eq:WeakForm3-a})-(\ref{eq:WeakForm3-d}). Moreover, there exists a constant $C$ independent of $t$ and $g$, such that
 \begin{equation}
 \|w\|_{2}+\|\bfbeta\|_{2}+\|r\|_2+\|p\|_1+t\|p\|_2\le C\|g\|_{0}.
\end{equation}
\end{theorem}

With the above theorem, we obtain some further results:
\begin{theorem} \label{Th:Regularity}
Let $(r,\bfbeta,p,w)$ be the solution of the problem (\ref{eq:WeakForm3-a})-(\ref{eq:WeakForm3-d}). Then the following three conclusions (i)-(iii) hold.

\noindent (i)  \ The quadruple
$(\bfM=- \mathcal{D}\bfepsilon(\bfbeta),\bfgamma=\grad r+\curl p,w,\bfbeta)\in \mathbb{M}\times \Gamma\times W\times
\Theta$ is the unique solution of the problem (\ref{eq:WeakForm2-a})-(\ref{eq:WeakForm2-b});

\noindent (ii) \ If $\bfM\in {\bf H}({\bfdiv};\Omega):=\{\bfQ\in L^{2}(\Omega)^{2\times2}_{sym}:\ {\bfdiv}\bfQ \in L^{2}(\Omega)^{2}\}$,
then the  equilibrium relation (\ref{1st-Euler}) holds;

\noindent (iii) \ Provided that $g\in L^{2}(\Omega)$, it holds
\begin{equation}\label{eq:Regularity}
 \|w\|_{2}+\|\bfbeta\|_{2}+\|\bfM\|_{1}+\|\bfgamma\|_{0}+t\|\bfgamma\|_{1}+\|r\|_2+\|p\|_1+t\|p\|_2\lesssim
\|g\|_{0}.
\end{equation}
\end{theorem}

\section{Finite element formulations for MiSP method}
This section is devoted to the finite element formulations of the
MiSP element on triangular and quadrilateral meshes. Let $\mathcal{T}_h$ be a
regular family of finite element subdivisions of the polygonal
domain $\Omega$. We denote by $h_{K}$  the diameter of a triangle or a quadrilateral $K\in \mathcal{T}_{h}$, and denote $h:=\max_{K\in \mathcal{T}_{h}}h_{K}$.

 Let $\mathbb{M}_{h}\subset \mathbb{M}$, $\Gamma_{h}\subset\Gamma$,
$W_{h}\subset W$, $\Theta_{h}\subset \Theta$
 be finite dimensional spaces for the
bending moment, shear stress, transverse displacement, and rotation
approximations. Then the corresponding finite element scheme for the
problem (\ref{eq:WeakForm2-a})-(\ref{eq:WeakForm2-b}) reads as: Find
$(\bfM_h,\bfgamma_h,w_h,\bfbeta_h)\in\mathbb{M}_h\times \Gamma_h
\times W_h\times \Theta_h$ such that
\begin{eqnarray}
a(\bfM_h,\bfgamma_h;\bfQ_h,\bftau_h)+\tilde{b}(\bfQ_h,\bftau_h;w_h,\bfbeta_h)=0
\ \ \mbox{~for all~} (\bfQ_h,\bftau_h)\in
\mathbb{M}_h\times\Gamma_h,\label{eq:DiscreteForm1-a}
\\
\tilde{b}(\bfM_h,\bfgamma_h;v_h,\bfzeta_h)=-\int_{\Omega}gv_hd\bfx
\ \ \mbox{~for all~} (v_h,\bfzeta_h)\in W_h\times\Theta_h,
\label{eq:DiscreteForm1-b}
\end{eqnarray}
where
\begin{equation}
 \tilde{b}(\bfQ_{h},\bftau_h;v_{h},\bfzeta_{h})
 :=\int_{\Omega}\bfQ_{h}:\bfepsilon(\bfzeta_{h})d\bfx-\sum\limits_{K\in\mathcal{T}_h}\int_{K}\bftau_{h}\cdot \bfR_h(\grad v_{h}-\bfzeta_{h})d\bfx,
\end{equation}
 and the reduction operator
 \begin{equation}\label{Rh}
 \bfR_h: H^{1}(\Omega)^{2}\bigcap H_{0}(rot,\Omega)\rightarrow Z_{h}
 \end{equation}
 is defined by \cite{Duran-Hernandez-Nieto-Liberman-Rodriguez2004}
 \begin{equation}
 \int_{e} \bfR_h\bfpsi\cdot\bft_e=\int_{e} \bfpsi\cdot\bft_e, \forall
 \mbox{ edge } e \mbox{ of } \mathcal{T}_h,
 \end{equation}
 where
 \begin{equation}\label{eq: HrotSpace}
  H_0(\rot,\Omega):=\{\bfpsi\in L^2(\Omega)^2: \rot\bfpsi\in L^2(\Omega), \bfpsi\cdot\bft|_{\partial \Omega}=0)\},
\end{equation}
$Z_h$ is to be defined in \eqref{Zh-MiSP3} for MiSP3 and in \eqref{Zh-MiSP4} for MiSP4, respectively, and $\bft_e$ denotes a unit vector tangent to $e$.

 For both MiSP3 and MiSP4 elements, we define
 \begin{equation}
 \Gamma_h=\bfdiv_h\mathbb{M}_h, \quad\text{with~}
 (\bfQ_h,\bftau_h)=(\bfQ_h,\bfdiv_h\bfQ_h)
\end{equation}
for $\bfQ_h\in\mathbb{M}_h$. Here $\bfdiv_h$ denotes the divergence
operator piecewise defined with respect to $\mathcal{T}_h$.

From the definition of the space $\Gamma_h$, we
have an equivalent form of the discrete scheme
(\ref{eq:DiscreteForm1-a})-(\ref{eq:DiscreteForm1-b}): Find
$(\bfM_h,w_h,\bfbeta_h)\in\mathbb{M}_h\times W_h\times \Theta_h$
such that
\begin{eqnarray}
\qquad
a(\bfM_h,\bfdiv_h\bfM_h;\bfQ_h,\bfdiv_h\bfQ_h)+\tilde{b}(\bfQ_h,\bfdiv_h\bfQ_h;w_h,\bfbeta_h)=0
\ \ \mbox{~for all~} \bfQ_h\in\mathbb{M}_h, \label{eq:DiscreteForm2-a}
\\
\tilde{b}(\bfM_h,\bfdiv_h\bfM_h;v_h,\bfzeta_h)=-\int_{\Omega}gv_hd\bfx\
\ \mbox{~for all~} (v_h,\bfzeta_h)\in W_h\times\Theta_h.
\label{eq:DiscreteForm2-b}
\end{eqnarray}

\subsection {Finite Dimensional Subspaces for MiSP3}
Let $\mathcal{T}_h$ be a conventional triangular mesh  of ${\Omega}$.
For element MiSP3, the continuous piecewise linear interpolation is used for the transverse displacement and rotation
approximation, i.e. the transverse displacement space $W_h$ and
rotation space $\Theta_h$ are chosen as
 \begin{equation}
  W_{h}:=\{v_{h}\in H_{0}^{1}(\Omega)\bigcap C(\bar{\Omega}): v_{h}|_{K}\in P_1(K) \mbox{ for all } K \in
  \mathcal{T}_{h}\},
 \end{equation}
  \begin{equation}
 \Theta_{h}:=\{\bfzeta_{h}\in (H_{0}^{1}(\Omega)\bigcap C(\bar{\Omega}))^{2}: \bfzeta_{h}|_{K}\in P_1(K)^{2} \mbox{ for all } K \in
 \mathcal{T}_{h}\}.
 \end{equation}
 Here $P_1(K)$ denotes the set of linear polynomials on $K$.

 For the approximation of bending moment tensor, we define
 \begin{equation}
 \mathbb{M}_{h}:=\{\bfQ_{h}\in L^{2}(\Omega)_{sym}^{2\times 2}: (\bfQ_{h}|_{K})_{i,j}\in P_1(K) \mbox{ for all } K \in
 \mathcal{T}_{h}, i,j=1,2\}.
 \end{equation}

We take the space $Z_h$ in \eqref{Rh} as
\begin{equation}\label{Zh-MiSP3}
Z_{h}:=\left\{\bfpsi_{h}\in H_{0}(\rot,\Omega):\bfpsi_{h}|_{K}=span\left\{
 \left(\begin{array}{c}
  1\\
  0
   \end{array}
   \right),
   \left(\begin{array}{c}
  0\\
  1
   \end{array}
   \right),
   \left(\begin{array}{c}
  y\\
  -x
   \end{array}
   \right)
   \right\},
  \mbox{ for all } K \in
 \mathcal{T}_{h}
  \right\}.
\end{equation}

We also need the space
 \begin{equation}
 P_{h}:=\{q_{h}\in L_0^{2}(\Omega): q_h|_{K}\in P_1(K) \mbox{ for all } K \in
 \mathcal{T}_{h}, q_h \mbox{ is continuous at midpoints of element edges}\}.
 \end{equation}

\subsection {Finite Dimensional Subspaces for MiSP4}
Let $\mathcal{T}_h$ be a conventional quadrilateral mesh  of ${\Omega}$.
Let $Z_{i}(x_{i},y_{i})$, $1\leq i\leq 4$ be
 the four vertices of $K$, and  $T_{i}$ be the sub-triangle
 of $K$ with vertices $Z_{i-1}$, $Z_{i}$ and $Z_{i+1}$
 (the index on $Z_{i}$ is modulo 4). Define
\begin{equation*}
 \rho_{K}=\min\limits_{1\leq i\leq4}\{\mathrm{diameter\ of\ circle\ inscribed\
 in}\ T_{i}\}.
\end{equation*}
Throughout the paper, we assume that the partition $\mathcal{T}_{h}$ satisfies the
following `shape-regularity' hypothesis: There exists a constant $\varrho>2$ independent
of $h$ such that for all $K\in \mathcal{T}_{h},$
\begin{equation}\label{partition condition}
 h_{K}\leq \varrho \rho_{K}.
\end{equation}

 Let $\hat{K}=[-1,1]\times[-1,1]$ be the reference square with
 vertices $\hat{Z}_{i}$, $1\leq i\leq 4$. For a quadrilateral $K\in \mathcal{T}_{h}$, there exists a unique
 invertible mapping $F_{K}$ that maps $\hat{K}$ onto
 $K$ with $F_{K}(\xi,\eta)\in Q_{1}^{2}(\xi,\eta)$ and $F_{K}(\hat{Z}_{i})=Z_{i}$, $1\leq i\leq 4$ (Figure \ref{fig:transfer}). Here $\xi, \eta\in [-1,1]$ are the local isoparametric coordinates.

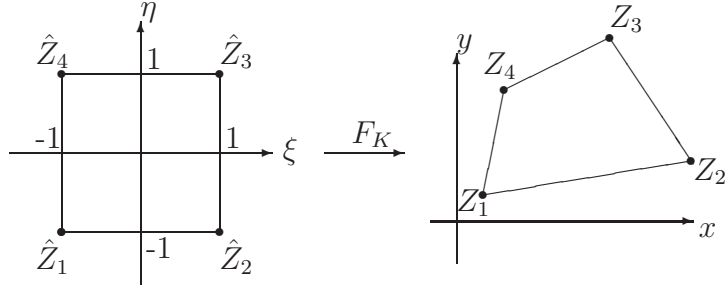
\begin{figure}[h]
\begin{center}
\setlength{\unitlength}{0.7cm}
\begin{picture}(10,6)
\put(0,1.5){\line(1,0){3}}        \put(3,1.5){\line(0,1){3}} \put(0,1.5){\line(0,1){3}} \put(0,4.5){\line(1,0){3}} \put(0,1.5){\circle*{0.15}}
\put(0,4.5){\circle*{0.15}} \put(3,4.5){\circle*{0.15}}       \put(0,1.5){\circle*{0.15}} \put(3,1.5){\circle*{0.15}}
\put(-0.5,0.8){\bf{$\hat{Z}_{1}$}} \put(3,0.8){\bf{$\hat{Z}_{2}$}}    \put(3,4.7){\bf{$\hat{Z}_{3}$}} \put(-0.5,4.7){\bf{$\hat{Z}_{4}$}}
\put(-1,3){\vector(1,0){5}} \put(1.5,0.5){\vector(0,1){5}} \put(4.2,2.9){$\xi$} \put(1.5,5.6){$\eta$}

\put(1.6,1.0){-1}                 \put(1.6,4.6){1} \put(-0.5,3.1){-1} \put(3.1,3.1){1}

\put(5,3){\vector(1,0){1.5}}      \put(5.5,3.2){$F_K$}

\put(8,2.2){\line(6,1){4}}        \put(8,2.2){\line(1,5){0.4}} \put(8.4,4.2){\line(2,1){2}} \put(10.4,5.2){\line(2,-3){1.6}}
\put(8,2.2){\circle*{0.15}}      \put(8.4,4.2){\circle*{0.15}} \put(10.4,5.2){\circle*{0.15}}    \put(11.95,2.85){\circle*{0.15}}
\put(7.5,1.9){\bf{$Z_{1}$}} \put(12,2.5){\bf{$Z_{2}$}} \put(10.4,5.4){\bf{$Z_{3}$}}       \put(8.0,4.5){\bf{$Z_{4}$}}
\put(7,1.7){\vector(1,0){5}}      \put(7.5,0.9){\vector(0,1){4}} \put(12.1,1.4){$x$} \put(7.5,5.0){$y$}

\end{picture}
\end{center}
\caption{The mapping $F_{K}$\label{fig:transfer}}
\end{figure}

 This  isoparametric bilinear mapping $(x,y)=F_{K}(\xi,\eta)$  is given by
\begin{equation}\label{relation1}
 x=\sum_{i=1}^{4}x_{i}N_{i}(\xi,\eta),\,\,\,
 y=\sum_{i=1}^{4}y_{i}N_{i}(\xi,\eta),
\end{equation}
 where
 $$N_{1}=\frac{1}{4}(1-\xi)(1-\eta),\,
   N_{2}=\frac{1}{4}(1+\xi)(1-\eta),\,
   N_{3}=\frac{1}{4}(1+\xi)(1+\eta),\,
   N_{4}=\frac{1}{4}(1-\xi)(1+\eta).$$
 We can rewrite (\ref{relation1}) as
\begin{equation}\label{relation2}
 x=a_{0}+a_{1}\xi+a_{2}\eta+a_{12}\xi\eta,\,\,\,
 y=b_{0}+b_{1}\xi+b_{2}\eta+b_{12}\xi\eta,
\end{equation}
with
\begin{equation*}
\left(\begin{array}{cc}
 a_{0} &b_{0}\\
 a_{1} &b_{1}\\
 a_{2} &b_{2}\\
 a_{12} &b_{12}
\end{array}\right)=\frac{1}{4}
\left(\begin{array}{cccc}
  1 &1  &1 &1\\
 -1 &1  &1 &-1\\
 -1 &-1 &1 &1\\
  1 &-1 &1 &-1\\
\end{array}\right)
\left(\begin{array}{cc}
 x_{1} &y_{1}\\
 x_{2} &y_{2}\\
 x_{3} &y_{3}\\
 x_{4} &y_{4}\\
\end{array}\right).
\end{equation*}
 The  Jacobi matrix and the Jacobian  of the transformation $F_{K}$ are respectively given by
 \begin{equation*}
 DF_{K}(\xi,\eta)=
 \left(\begin{array}{cc}
  \frac{\partial x}{\partial \xi} &\frac{\partial x}{\partial \eta}\\
  \frac{\partial y}{\partial \xi} &\frac{\partial y}{\partial \eta}\\
 \end{array}\right)=
 \left(\begin{array}{cc}
  a_{1}+a_{12}\eta &a_{2}+a_{12}\xi\\
  b_{1}+b_{12}\eta &b_{2}+b_{12}\xi\\
 \end{array}\right),
 \end{equation*}
 \begin{equation*}
  J_{K} =det(DF_{K})=J_{0}+J_{1}\xi+J_{2}\eta,
 \end{equation*}
 where
 \begin{equation*}
  J_{0}=a_{1}b_{2}-a_{2}b_{1},\,J_{1}=a_{1}b_{12}-a_{12}b_{1},\,J_{2}=a_{12}b_{2}-a_{2}b_{12}.
 \end{equation*}

 \begin{remark} Notice that when $K$ is a parallelogram,  we have $a_{12}=b_{12}=0$, and  $F_{K}$ is reduced to an affine mapping. Especially, when $K$ is a rectangle, we further have $a_{2}=b_{1}=0$.
 \end{remark}

For element MiSP4, the continuous isoparametric bilinear
interpolation is used for the transverse displacement and rotation
approximation, i.e. the transverse displacement space $W_h$ and
rotation space $\Theta_h$ are chosen as
 \begin{equation}
 W_{h}:=\{v_{h}\in H_{0}^{1}(\Omega)\bigcap C(\bar{\Omega}): v_{h}|_{K}\circ F_{K}\in Q_1(\hat{K}) \mbox{ for all } K \in
 \mathcal{T}_{h}\},
 \end{equation}
  \begin{equation}
 \Theta_{h}:=\{\bfzeta_{h}\in (H_{0}^{1}(\Omega)\bigcap C(\bar{\Omega}))^{2}: \bfzeta_{h}|_{K}\circ F_{K}\in Q_1(\hat{K})^{2} \mbox{ for all } K \in
 \mathcal{T}_{h}\}.
 \end{equation}
Here $Q_1(\hat{K})$ denotes the set of bilinear polynomials on $\hat{K}$.
 For the approximation of bending moment tensor, we define
 \begin{equation}
 \mathbb{M}_{h}:=\{\bfQ_{h}\in L^{2}(\Omega)_{sym}^{2\times 2}: (\bfQ_{h}|_{K}\circ F_{K})_{i,j}\in Q_1(\hat{K}) \mbox{ for all } K \in
 \mathcal{T}_{h}, i,j=1,2\}.
 \end{equation}

We take the space $Z_h$ in \eqref{Rh} as
\begin{equation}\label{Zh-MiSP4}
Z_{h}:=\{\bfpsi_{h}\in H_{0}(\rot,\Omega):\bfpsi_{h}|_{K}\circ
F_{K}=span\{DF_{K}^{-t}
 \left(\begin{array}{llll}
  1 &\eta &0 &0\\
  0 &0 &1 &\xi\\
   \end{array}
   \right)\},
  \mbox{ for all } K \in
 \mathcal{T}_{h}
  \}.
\end{equation}

\section{error analysis for MiSP3}
 In this section we will derive  error estimates for the MiSP3 element. The corresponding subspaces in this section
 are defined as in subsection 3.1. We first give the following properties for the operator $\bfR_h$.
 \begin{lemma}\label{lemma: error_Rh_MiSP3}
 The operator $\bfR_h: H^{1}(\Omega)^{2}\bigcap H_{0}(rot,\Omega)\rightarrow Z_{h}$ satisfies
 \begin{equation}\label{eq: Rh-Property-1}
 \bfR_h(\grad v_{h})=\grad v_{h}, \forall v_h\in W_h,
 \end{equation}
 \begin{equation}\label{eq: Rh-Property-2}
 \|\bfeta-\bfR_h\bfeta\|_{0}\lesssim h\|\bfeta\|_{1}, \forall \bfeta\in H^{1}(\Omega)^{2}\bigcap
 H_{0}(\rot,\Omega),
 \end{equation}
 \begin{equation}\label{eq: Rh-Property-3}
 \|\rot(\bfR_h\bfeta)\|_0\lesssim  \|\bfeta\|_1, \forall \bfeta\in H^{1}(\Omega)^{2}\bigcap
 H_{0}(\rot,\Omega),
 \end{equation}
 \begin{equation}\label{eq: Rh-Property-4}
 \rot(\bfR_h\bfeta_h)=\rot(\bfeta_h), \forall \bfeta_h\in \Theta_h.
 \end{equation}
\end{lemma}
\begin{proof}\
It is easy to verify $\grad W_h\subset Z_h$ and $\bfR_h\bfpsi_h=\bfpsi_h, \forall \bfpsi_h\in Z_h$. Then \eqref{eq: Rh-Property-1} holds.
The estimate \eqref{eq: Rh-Property-2} follows from a scaling argument and the definition of $\bfR_h$.

For $\bfeta\in H^{1}(\Omega)^{2}\bigcap
 H_{0}(\rot,\Omega)$, let $\Pi_h\bfeta$ be the Scott-Zhang interpolation \cite{scott1990finite} of $\bfeta$. Then we have
\begin{eqnarray*}
  \|\rot(\bfR_h\bfeta)\|_0
  &\le &\|\rot(\bfR_h\bfeta)-\rot(\Pi_h\bfeta)\|_0+\|\rot(\Pi_h\bfeta)\|_0 \\
  &\lesssim &h^{-1}\|\bfR_h\bfeta-\Pi_h\bfeta\|_0+\|\bfeta\|_1\\
  &\le &h^{-1}(\|\bfR_h\bfeta-\bfeta\|+\|\bfeta-\Pi_h\bfeta\|_0)+\|\bfeta\|_1\\
  &\lesssim &\|\bfeta\|_1.
\end{eqnarray*}
Here, the second inequality is based on an inverse inequality and the stability of Scott-Zhang interpolation. Hence \eqref{eq: Rh-Property-3} holds.

For any $K\in\mathcal{T}_h$, it holds
\begin{equation*}
\int_K \rot(\bfR_h\bfeta_h)d\bfx=\int_{\partial K} \bfR_h\bfeta_h\cdot \bft ds=\int_{\partial K} \bfeta_h\cdot \bft ds=\int_K \rot(\bfeta_h)d\bfx,
\end{equation*}
since $\rot(\bfeta_h)|_K$ and $\rot(\bfR_h\bfeta_h)|_K$ are constants, we have $\rot(\bfeta_h)|_K=\rot(\bfR_h\bfeta_h)|_K$, which yields \eqref{eq: Rh-Property-4}.
\end{proof}

 For the latter error analysis, we need the following discrete Helmholtz decomposition given in Theorem 4.1 of \cite{xiaoliang1994simple}.
\begin{lemma}\label{lemma: DiscreteDecomposition}
 For any $\bfQ_h\in \mathbb{M}_{h}$, there exist $s_h\in W_h$ and $q_h\in P_h$ such that
\begin{equation}\label{eq: DiscreteDecomposition}
\bfdiv_h \bfQ_h=\grad s_h + \curlh q_h.
\end{equation}
\end{lemma}

In the latter analysis, we will use the discrete Helmholtz decomposition \eqref{eq: DiscreteDecomposition} for  $\bfQ\in\mathbb{M}_{h}$ and the Helmholtz decomposition \eqref{eq:Helmholtz} with $\bftau=\bfdiv \bfQ$ for
$\bfQ\in (H^1(\Omega))^{2\times 2}_{sym}$ respectively. For convenience, we denote
the decomposition as $\bfdiv_h \bfQ=\grad s + \curlh q$ in both cases.

We introduce two mesh-dependent norms as follows:
for any $\bfQ\in (H^1(\Omega))^{2\times 2}_{sym}\bigcup\mathbb{M}_h$, $v\in (H^2(\Omega)\bigcap H_0^1(\Omega))\bigcup W_h$, $\bfzeta\in H_0^1(\Omega)^2\bigcup \Theta_h$, \begin{equation}
|\|\bfQ|\|_{h,1}:=\|\bfQ\|_0+(h+t)\|\curlh q\|_0+\|q\|_0+\|\grad s\|_0,
\end{equation}
\begin{equation}
|\|(v,\bfzeta)|\|_{h,2}:=\|\bfepsilon(\bfzeta)\|_0+\|\bfR_h(\grad v)\|_0.
\end{equation}

We are now ready to give the error analysis. Basing on the standard error theory for mixed methods,
we first show   continuity results in Lemmas \ref{lemma: Continuity1-MiSP3}-\ref{lemma: Continuity2-MiSP3},
then derive   coercivity results in Lemmas \ref{lemma: Coercivity1-MiSP3}-\ref{lemma: Coercivity2-MiSP3},
we finally give the desired estimates in Theorem \ref{theorem: error-MiSP3}.
\begin{lemma}\label{lemma: Continuity1-MiSP3}
 It holds
 \begin{equation}
 a(\bfM,\bfdiv_h\bfM;\bfQ,\bfdiv_h\bfQ)\lesssim |\|\bfM|\|_{h,1}|\|\bfQ|\|_{h,1}\quad \mbox{for all~} \bfM, \bfQ\in \mathbb{M}\bigcup\mathbb{M}_h.
 \end{equation}
\end{lemma}
\begin{proof}
 It is trivial.
\end{proof}

\begin{lemma}\label{lemma: Continuity2-MiSP3}
 For any $\bfQ\in (H^1(\Omega))^{2\times 2}_{sym}\bigcup\mathbb{M}_h$, $v\in (H^2(\Omega)\bigcap H_0^1(\Omega))\bigcup W_h$, $\bfzeta\in H_0^1(\Omega)^2\bigcup \Theta_h$, it holds
 \begin{equation}
 \tilde{b}(\bfQ,\bfdiv_h\bfQ; v,\bfzeta)\lesssim |\|\bfQ|\|_{h,1}|\|(v,\bfzeta)|\|_{h,2}.
 \end{equation}
\end{lemma}
\begin{proof}
Given $\bfQ\in (H^1(\Omega))^{2\times 2}_{sym}\bigcup\mathbb{M}_h$, by \eqref{eq:Helmholtz} and \eqref{eq: DiscreteDecomposition},
we have $$\bfdiv_h\bfQ=\grad s+\curlh q.$$
We first show
\begin{equation}\label{eq: OrthogonalProperty}
(\curlh q,\bfR_h(\grad v))=0
\end{equation}
 holds.  By integration by parts, we have
 \begin{eqnarray}
  (\curlh q,\bfR_h(\grad v))&=&\sum_{K\in\mathcal{T}_h}\left(-(q,\rot(\bfR_h(\grad v)))_K+\int_{\partial K} q \bfR_h(\grad v)\cdot\bft ds\right)\nonumber\\
  &=&-\sum_{K\in\mathcal{T}_h}(q,\rot(\bfR_h(\grad v)))_K+\sum_{e\in \varepsilon_h}\int_{e} [q] \bfR_h(\grad v)\cdot\bft ds,\label{eq: IntegrationByParts}
 \end{eqnarray}
 here, $\varepsilon_h$ denotes the set of interior edges for $\mathcal{T}_h$, and $[q]|_e$ means the jump across the edge $e$.
 We only need to verify the two terms of \eqref{eq: IntegrationByParts} both vanish.

 Since $\rot(\bfR_h(\grad v))$ is a piecewise constant, and, for any $K\in\mathcal{T}_h$,
    \begin{equation*}
      \int_K \rot(\bfR_h(\grad v))d\bfx=\int_{\partial K} \bfR_h(\grad v)\cdot \bft ds=\int_{\partial K} (\grad v)\cdot \bft ds=0,
    \end{equation*}
    we have $\rot(\bfR_h(\grad v))|_K=0$.
    So, the first term of \eqref{eq: IntegrationByParts} equals zero.

 For the second term, if $\bfQ\in (H^1(\Omega))^{2\times 2}_{sym}$, it equals zero by continuity.
 Otherwise if $\bfQ\in \mathbb{M}_h$, since $q\in P_h$, $[q]$ vanishes at the midpoint of $e$ and $[q] \bfR_h(\grad v)\cdot\bft|_{e}$ is linear,
 then by one-point Gauss integration we know the second term equals zero.

%
%
%
%

Now with \eqref{eq: OrthogonalProperty}, we can deduce the desired result:
\begin{eqnarray*}
\tilde{b}(\bfQ,\bfdiv_h\bfQ;v,\bfzeta)
&=&(\bfQ,\bfepsilon(\bfzeta))-(\bfdiv_h\bfQ,\bfR_h(\grad v-\bfzeta))\\
&=&(\bfQ,\bfepsilon(\bfzeta))-(\grad s+\curlh q,\bfR_h(\grad v-\bfzeta))\\
&=&(\bfQ,\bfepsilon(\bfzeta))-(\grad s,\bfR_h(\grad v-\bfzeta))-(q,\rot(\bfR_h\bfzeta))\\
&=&\|\bfQ\|_0\|\bfepsilon(\bfzeta)\|_0+\|\grad s\|_0\|\bfR_h(\grad v-\bfzeta)\|_0+\|q\|_0\|\rot(\bfR_h\bfzeta)\|_0\\
&\lesssim&\left(\|\bfQ\|_0+\|\grad s\|_0+\|q\|_0\right)\left(\|\bfepsilon(\bfzeta)\|_0+\|\bfR_h(\grad v-\bfzeta)\|_0+\|\bfzeta\|_1\right)  \\
&\lesssim&|\|\bfQ|\|_{h,1}|\|(v,\bfzeta)|\|_{h,2}.
\end{eqnarray*}
\end{proof}

\begin{lemma}\label{lemma: Coercivity1-MiSP3}
 It holds
 \begin{equation}\label{eq: coercivity}
 a(\bfQ_h,\bfdiv_h\bfQ_h;\bfQ_h,\bfdiv_h\bfQ_h)\gtrsim |\|\bfQ_h|\|_{h,1}^2, \mbox{ for all } \bfQ_h\in Ker B,
 \end{equation}
 here,
 \begin{equation}
 Ker B=\left\{\bfQ_h\in\mathbb{M}_h: \tilde{b}(\bfQ_h,\bfdiv_h\bfQ_h;v_h,\bfzeta_h)=0, \mbox{ for all } (v_h,\bfzeta_h)\in W_h\times\Theta_h\right\}.
 \end{equation}
\end{lemma}
\begin{proof}
 We want to check the property of $\bfQ_h\in Ker B$. Based on Lemma \ref{lemma: DiscreteDecomposition},
 there exist $s_h\in W_h$ and $q_h\in P_h$ such that
 \begin{equation}\label{eq: DiscreteDecomposition2}
\bfdiv_h \bfQ_h=\grad s_h + \curlh q_h.
\end{equation}
It is easy to have
 \begin{equation*}
 a(\bfQ_h,\bfdiv_h\bfQ_h;\bfQ_h,\bfdiv_h\bfQ_h)\gtrsim \|\bfQ_h\|_0^2+t^2\|\bfdiv_h\bfQ_h\|_0^2.
 \end{equation*}
 By the inverse inequality $ \|\bfQ_h\|_0^2\gtrsim h^2\|\bfdiv_h\bfQ_h\|_0^2$ and the relation $\|\bfdiv_h\bfQ_h\|_0^2=\|\grad s_h\|_0^2+\|\curlh q_h\|_0^2$, we have
 \begin{equation}\label{eq: 1}
 a(\bfQ_h,\bfdiv_h\bfQ_h;\bfQ_h,\bfdiv_h\bfQ_h)\gtrsim \|\bfQ_h\|_0^2+(t+h)^2\|\curlh q_h\|_0^2.
 \end{equation}

 We next need to bound $\|\grad s_h\|_0$ and $\|q_h\|_0$.
 For any $(v_h,\bfzeta_h)\in W_h\times\Theta_h$, it holds
\begin{equation*}
\tilde{b}(\bfQ_h,\bfdiv_h\bfQ_h;v_h,\bfzeta_h)=(\bfQ_h,\bfepsilon(\bfzeta_h))-(\grad s_h,\grad v_h-\bfR_h\bfzeta_h)+(\curlh q_h,\bfR_h\bfzeta_h)=0.
 \end{equation*}

 On one hand, choose $\bfzeta_h=0$ and $v_h=s_h$, then $(\grad s_h,\grad s_h)=0$. Since $s_h\in H_0^1(\Omega)$, we have $s_h=0$.

 On the other hand, choose $v_h=0$, then
 \begin{eqnarray*}
(\bfQ_h,\bfepsilon(\bfzeta_h))+(\curlh q_h,\bfR_h\bfzeta_h)\nonumber
&=&(\bfQ_h,\bfepsilon(\bfzeta_h))-(q_h,\rot(\bfR_h\bfzeta_h))\\
&=&(\bfQ_h,\bfepsilon(\bfzeta_h))-(q_h,\rot(\bfzeta_h))=0,\mbox{ for any } \bfzeta_h\in \Theta_h.
 \end{eqnarray*}
 For the above $q_h$, there exists $\bfzeta\in H_0^1(\Omega)^2$, such that
 \begin{equation*}
\rot\bfzeta=q_h, \mbox{ and } \|\bfzeta\|_1\lesssim\|q_h\|_0.
 \end{equation*}
So we get
\begin{equation*}
\|q_h\|_0\lesssim\frac{(q_h,\rot\bfzeta)}{\|\bfzeta\|_1}
         =\frac{(q_h,\rot(\Pi_h\bfzeta))}{\|\bfzeta\|_1}
          +\frac{(q_h,\rot(\bfR_h\bfzeta-\Pi_h\bfzeta))}{\|\bfzeta\|_1}
          +\frac{(q_h,\rot(\bfzeta-\bfR_h\bfzeta))}{\|\bfzeta\|_1}.
\end{equation*}
For the first term in the right-hand side of this relation, it holds
\begin{equation*}
\frac{(q_h,\rot(\Pi_h\bfzeta))}{\|\bfzeta\|_1}=\frac{(\bfQ_h,\bfepsilon(\Pi_h\bfzeta))}{\|\bfzeta\|_1}
\lesssim\frac{(\bfQ_h,\bfepsilon(\Pi_h\bfzeta))}{\|\Pi_h\bfzeta\|_1}\le\|\bfQ_h\|_0.
\end{equation*}
For the second term, it holds
\begin{eqnarray*}
&&\frac{(q_h,\rot(\bfR_h\bfzeta-\Pi_h\bfzeta))}{\|\bfzeta\|_1}=-\frac{(\curlh q_h,\bfR_h(\bfzeta-\Pi_h\bfzeta))}{\|\bfzeta\|_1}\\
&=&-\frac{(\curlh q_h,(\bfR_h\bfzeta-\bfzeta)+(\bfzeta-\Pi_h\bfzeta)+(\Pi_h\bfzeta-\bfR_h(\Pi_h\bfzeta)))}{\|\bfzeta\|_1}\\
&\lesssim&h\|\curlh q_h\|_0.
\end{eqnarray*}
For the third term, it holds
\begin{eqnarray*}
&&\frac{(q_h,\rot(\bfzeta-\bfR_h\bfzeta))}{\|\bfzeta\|_1}=\frac{\Sigma_{K\in \mathcal{T}_{h}}(q_h,\rot(\bfzeta-\bfR_h\bfzeta))_K}{\|\bfzeta\|_1}\\
&=&\frac{\Sigma_{K\in \mathcal{T}_{h}}(q_h-q_0,\rot(\bfzeta-\bfR_h\bfzeta))_K}{\|\bfzeta\|_1} \quad(\mbox{here } q_0=\frac1{|K|}\int_K q_h d\bfx)\\
&\lesssim&h\|\curlh q_h\|_0.
\end{eqnarray*}
So, for $\bfQ_h\in Ker B$ with the decomposition \eqref{eq: DiscreteDecomposition2}, we have $s_h=0$ and $\|q_h\|_0\lesssim\|\bfQ_h\|_0+h\|\curlh q_h\|_0$,
which, together with \eqref{eq: 1}, imply the coercivity \eqref{eq: coercivity}.
\end{proof}

\begin{lemma}\label{lemma: Coercivity2-MiSP3}
 The inf-sup condition
\begin{equation}\label{inf-sup-MiSP3}
\sup_{\bfQ_{h}\in\mathbb{M}_{h}}
\frac{\tilde{b}(\bfQ_{h},\bfdiv_h\bfQ_h;v_{h},\bfzeta_{h})}{|\|\bfQ_{h}|\|_{h,1}}\gtrsim|\|(v_{h},\bfzeta_{h})|\|_{h,2},
\mbox{ for all } (v_{h},\bfzeta_{h})\in W_{h}\times \Theta_{h}
\end{equation}
holds.
\end{lemma}
\begin{proof}
Given $\bfzeta_h\in\Theta_h$, let $\bfQ_h^1=C_1\bfepsilon(\bfzeta_h)$ (the constant $C_1$ to be determined), then $\bfdiv_h \bfQ_h^1=0$.
Given $v_h\in W_h$, there exists $\bfQ_h^2\in\mathbb{M}_h$, such that
$\bfdiv_h\bfQ_h^2=-\grad v_h \mbox{~and~} \|\bfQ_h^2\|_0\le C_2 h\|\bfdiv_h\bfQ_h^2\|_0.$
Suppose $\|\bfR_h\bfzeta_h\|_0\le C_3\|\bfepsilon(\bfzeta_h)\|_0$.

Take $\bfQ_h=\bfQ_h^1+\bfQ_h^2$, then
\begin{eqnarray*}
&&\tilde{b}(\bfQ_{h},\bfdiv_h\bfQ_h;v_{h},\bfzeta_{h})\\
&=&(\bfQ_h,\bfepsilon(\bfzeta_h))-(\bfdiv_h\bfQ_h,\grad v_h-\bfR_h\bfzeta_h)\\
&=&(\bfQ_h^1,\bfepsilon(\bfzeta_h))+(\bfQ_h^2,\bfepsilon(\bfzeta_h))+(\grad v_h,\grad v_h)-(\grad v_h,\bfR_h\bfzeta_h)\\
&\ge& C_1\|\bfepsilon(\bfzeta_h)\|_0^2-C_2 h\|\grad v_h\|_0\|\bfepsilon(\bfzeta_h)\|_0+\|\grad v_h\|_0^2-C_3\|\grad v_h\|_0\|\bfepsilon(\bfzeta_h)\|_0\\
&\ge& \left(C_1-C_2^2h^2-C_3^2\right)\|\bfepsilon(\bfzeta_h)\|_0^2+\frac{1}{2}\|\grad v_h\|_0^2.
\end{eqnarray*}
Let $C_1\ge 2(C_2^2h^2+C_3^2)$, then
\begin{eqnarray*}
\tilde{b}(\bfQ_{h},\bfdiv_h\bfQ_h;v_{h},\bfzeta_{h})\ge \frac{C_1}{2}\|\bfepsilon(\bfzeta_h)\|_0^2+\frac{1}{2}\|\grad v_h\|_0^2\gtrsim|\|(v_{h},\bfzeta_{h})|\|_{h,2}^2.
\end{eqnarray*}
On the other hand,
\begin{eqnarray*}
|\|\bfQ_h|\|_{h,1}&=&|\|\bfQ_h^1+\bfQ_h^2|\|_{h,1}=\|\bfQ_h^1+\bfQ_h^2\|_0+\|\grad v_h\|_0\\
&\le&\|\bfQ_h^1\|_0+\|\bfQ_h^2\|_0+\|\grad v_h\|_0\\
&\le&C_1\|\bfepsilon(\bfzeta_h)\|_0+(C_2h+1)\|\grad v_h\|_0\lesssim |\|(v_{h},\bfzeta_{h})|\|_{h,2}.
\end{eqnarray*}
Then the result (\ref{inf-sup-MiSP3}) holds.
\end{proof}

 \begin{theorem}\label{theorem: error-MiSP3}
 Let $(\bfM,\bfgamma=\bfdiv\bfM,w,\bfbeta)\in \mathbb{M}\times\Gamma\times W\times \Theta$ be the
 solution of the problem \eqref{eq:WeakForm2-a}-\eqref{eq:WeakForm2-b}. Then the discretization
 problem \eqref{eq:DiscreteForm2-a}-\eqref{eq:DiscreteForm2-b} admits a unique solution
 $(\bfM_{h},w_{h},\bfbeta_{h})\in \mathbb{M}_{h}\times W_{h}\times\Theta_{h}$ such that
\begin{eqnarray*}
&&|\|\bfM-\bfM_{h}|\|_{h,1}+|\|(w-w_{h},\bfbeta-\bfbeta_{h})|\|_{h,2}\\
& \lesssim &\inf\limits_{\bfQ_{h}\in
\mathbb{M}_{h}}|\|\bfM-\bfQ_{h}|\|_{h,1}+\inf\limits_{(v_{h},\bfzeta_{h})\in W_h\times\Theta_h}|\|(w-v_{h},\bfbeta-\bfzeta_{h})|\|_{h,2}+ht\|\bfgamma\|_{1}+h\|\bfgamma\|_{0}.\\
\end{eqnarray*}
\end{theorem}

\begin{proof}
 Since
\begin{equation*}
a(\bfM,\bfgamma;\bfQ_{h},\bfdiv_h\bfQ_h)+\tilde{b}(\bfQ_{h},\bfdiv_h\bfQ_h;w,\bfbeta)-(\bfdiv_h\bfQ_h,\grad
w-\bfbeta-\bfR_h(\grad w-\bfbeta))=0, \forall
\bfQ_{h}\in\mathbb{M}_{h},
\end{equation*}
\begin{equation*}
 a(\bfM_{h},\bfdiv_h\bfM_h;\bfQ_{h},\bfdiv_h\bfQ_h)+\tilde{b}(\bfQ_{h},\bfdiv_h\bfQ_h;w_{h},\bfbeta_{h})=0, \forall \bfQ_{h}\in\mathbb{M}_{h},
\end{equation*}
 then for all $\bfQ_{h}\in\mathbb{M}_{h}$, it holds
 \begin{eqnarray*}
 a(\bfM-\bfM_{h},\bfgamma-\bfdiv_h\bfM_h;\bfQ_{h},\bfdiv_h\bfQ_h)+\tilde{b}(\bfQ_{h},\bfdiv_h\bfQ_h;w-w_{h},\bfbeta-\bfbeta_{h})\\
 -(\bfdiv_h\bfQ_h,\grad
 w-\bfbeta-\bfR_h(\grad w-\bfbeta))=0.
 \end{eqnarray*}

Denote
\begin{equation*}
Z_{h}(g)=\{\bfQ_{h}\in\mathbb{M}_{h}:
\tilde{b}(\bfQ_{h},\bfdiv_h\bfQ_h;v_{h},\bfzeta_{h})=-(g,v_{h}),\ \forall
(v_{h},\bfzeta_{h})\in W_{h}\times\Theta_{h}\}.
\end{equation*}
Let $\tilde{\bfQ}_{h}$ be any element of $Z_{h}(g)$. Since
$\tilde{\bfQ}_{h}-\bfM_{h}\in Z_h(0)=Ker B$, then
\begin{equation*}
\begin{array}{ll}
&|\|\tilde{\bfQ}_{h}-\bfM_{h}|\|_{h,1}^{2}\\\vspace{0.1cm}
\lesssim
&a(\tilde{\bfQ}_{h}-\bfM_{h},\bfdiv_h(\tilde{\bfQ}_{h}-\bfM_{h});\tilde{\bfQ}_{h}-\bfM_{h},\bfdiv_h(\tilde{\bfQ}_{h}-\bfM_{h}))\\ \vspace{0.1cm}
=&a(\tilde{\bfQ}_{h}-\bfM,\bfdiv_h\tilde{\bfQ}_h-\bfgamma;\tilde{\bfQ}_{h}-\bfM_h,\bfdiv_h(\tilde{\bfQ}_h-\bfM_h))\\\vspace{0.1cm}
  &+a(\bfM-\bfM_{h},\bfgamma-\bfdiv_h\bfM_{h};\tilde{\bfQ}_{h}-\bfM_{h},\bfdiv_h(\tilde{\bfQ}_{h}-\bfM_{h}))\\\vspace{0.1cm}
=&a(\tilde{\bfQ}_{h}-\bfM,\bfdiv_h(\tilde{\bfQ}_h-\bfM);\tilde{\bfQ}_{h}-\bfM_{h},\bfdiv_h(\tilde{\bfQ}_{h}-\bfM_{h}))\\\vspace{0.1cm}
  &-\tilde{b}(\tilde{\bfQ}_{h}-\bfM_{h},\bfdiv_h(\tilde{\bfQ}_h-\bfM_h);w-w_{h},\bfbeta-\bfbeta_{h})\\\vspace{0.1cm}
&+(\bfdiv_h(\tilde{\bfQ}_h-\bfM_h),\grad w-\bfbeta-\bfR_h(\grad w-\bfbeta))\\\vspace{0.1cm}
\lesssim
 &|\|\tilde{\bfQ}_{h}-\bfM_h|\|_{h,1}(|\|\tilde{\bfQ}_{h}-\bfM|\|_{h,1}
 +|\|(w-v_{h},\bfbeta-\bfzeta_{h})|\|_{h,2}+ht\|\bfgamma\|_{1}).
\end{array}
\end{equation*}
So we have
\begin{equation*}
|\|\tilde{\bfQ}_{h}-\bfM_{h}|\|_{h,1}
\lesssim
 |\|\tilde{\bfQ}_{h}-\bfM|\|_{h,1}+|\|(w-v_{h},\bfbeta-\bfzeta_{h})|\|_{h,2}+ht\|\bfgamma\|_{1}.
\end{equation*}
Then, by using the triangle inequality, we get
\begin{equation}\label{eq: 2}
|\|\bfM-\bfM_{h}|\|_{h,1}
\lesssim
 \inf\limits_{\tilde{\bfQ}_{h}\in Z_{h}(g)}|\|\tilde{\bfQ}_{h}-\bfM|\|_{h,1}
 +\inf\limits_{(v_{h},\bfzeta_{h})\in
 W_h\times\Theta_h}|\|(w-v_{h},\bfbeta-\bfzeta_{h})|\|_{h,2}+ht\|\bfgamma\|_{1}.
\end{equation}

For any $\bfQ_{h}\in \mathbb{M}_{h}$, there
exists $\bar{\bfQ}_{h}\in\mathbb{M}_{h}$, such that, for all
$(v_{h},\bfzeta_{h})\in W_h\times\Theta_h$,
\begin{equation*}
\tilde{b}(\bar{\bfQ}_{h},\bfdiv_h\bar{\bfQ}_{h};v_{h},\bfzeta_{h})=\tilde{b}(\bfM-\bfQ_{h},\bfdiv_h(\bfM-\bfQ_{h});v_{h},\bfzeta_{h})-(\bfgamma,\grad
v_{h}-\bfzeta_{h}-\bfR_h(\grad v_{h}-\bfzeta_{h}))
\end{equation*}
and
\begin{eqnarray}
 \nonumber
 &&|\|\bar{\bfQ}_{h}|\|_{h,1}\\
 \nonumber
 &\lesssim
 &\sup\limits_{(v_{h},\bfzeta_{h})\in W_h\times\Theta_h}\frac{\tilde{b}(\bar{\bfQ}_{h},\bfdiv_h\bar{\bfQ}_{h};v_{h},\bfzeta_{h})}{|\|(v_{h},\bfzeta_{h})|\|_{h,2}}\\
 \nonumber
 &\lesssim
 &\sup\limits_{(v_{h},\bfzeta_{h})\in W_h\times\Theta_h}\frac{\tilde{b}(\bfM-\bfQ_{h},\bfdiv_h(\bfM-\bfQ_h);v_{h},\bfzeta_{h})
  -(\bfgamma,\grad v_{h}-\bfzeta_{h}-\bfR_h(\grad v_{h}-\bfzeta_{h}))}{|\|(v_{h},\bfzeta_{h})|\|_{h,2}}\\
 \nonumber
 &=
 &\sup\limits_{(v_{h},\bfzeta_{h})\in W_h\times\Theta_h}\frac{\tilde{b}(\bfM-\bfQ_{h},\bfdiv_h(\bfM-\bfQ_h);v_{h},\bfzeta_{h})
  +(\bfgamma,\bfzeta_{h}-\bfR_h\bfzeta_{h})}{|\|(v_{h},\bfzeta_{h})|\|_{h,2}}\\
 \nonumber
 &\lesssim
 &|\|\bfM-\bfQ_{h}|\|_{h,1}+h\|\bfgamma\|_{0}.
\end{eqnarray}
Choose $\tilde{\bfQ}_{h}=\bar{\bfQ}_{h}+\bfQ_{h}$, then
$\tilde{\bfQ}_{h}\in Z_{h}(g)$. Thus we get
\begin{equation*}
\nonumber
|\|\bfM-\tilde{\bfQ}_{h}|\|_{h,1}
=|\|\bfM-\bfQ_{h}-\bar{\bfQ}_{h}|\|_{h,1}
\leq|\|\bfM-\bfQ_{h}|\|_{h,1}+|\|\bar{\bfQ}_{h}|\|_{h,1}
\lesssim
|\|\bfM-\bfQ_{h}|\|_{h,1}+h\|\bfgamma\|_{0}.
\end{equation*}
This estimate and \eqref{eq: 2} imply
\begin{equation}\label{eq: 3}
 |\|\bfM-\bfM_{h}|\|_{h,1}
 \lesssim
 \inf_{\bfQ_{h}\in \mathbb{M}_{h}}|\|\bfM-\bfQ_{h}|\|_{h,1}
 +\inf_{(v_{h},\bfzeta_{h})\in W_h\times\Theta_h}|\|(w-v_{h},\bfbeta-\bfzeta_{h})|\|_{h,2}
 +ht\|\bfgamma\|_{1}+h\|\bfgamma\|_{0}.
\end{equation}

On the other hand, from the coercivity and continuity properties we get
\begin{eqnarray}
\nonumber
 &&|\|(v_{h}-w_{h},\bfzeta_{h}-\bfbeta_{h})|\|_{h,2}\\
 \nonumber
  \displaystyle&\lesssim
 &\sup\limits_{\bfQ_{h}\in\mathbb{M}_{h}}\frac{\tilde{b}(\bfQ_{h},\bfdiv_h\bfQ_h;v_{h}-w_{h},\bfzeta_{h}-\bfbeta_{h})}{|\|\bfQ_{h}|\|_{h,1}}\\
 \nonumber
  \displaystyle&=&\sup\limits_{\bfQ_{h}\in\mathbb{M}_{h}}\left\{\frac{ -a(\bfM-\bfM_{h},\bfdiv_h(\bfM-\bfM_{h});\bfQ_{h},\bfdiv_h\bfQ_h)-\tilde{b}(\bfQ_{h},\bfdiv_h\bfQ_h;w-v_{h},\bfbeta-\bfzeta_{h})}{|\|\bfQ_{h}|\|_{h,1}}\right.\\
 \nonumber
 \displaystyle &&\left.\qquad\qquad\qquad+\frac{
  (\bfdiv_h\bfQ_h,\grad w-\bfbeta-\bfR_h(\grad w-\bfbeta))}{|\|\bfQ_{h}|\|_{h,1}}\right\}\\
  \displaystyle\nonumber
 &\lesssim
 &|\|\bfM-\bfM_{h}|\|_{h,1}+|\|(w-v_{h},\bfbeta-\bfzeta_{h})|\|_{h,2}+ht\|\bfgamma\|_{1}.
\end{eqnarray}
This inequality and \eqref{eq: 3} imply
\begin{eqnarray}
\label{eq: 4}&&|\|(w-w_{h},\bfbeta-\bfbeta_{h})|\|_{h,2} \\
\nonumber &\lesssim &\inf_{\bfQ_{h}\in
\mathbb{M}_{h}}|\|\bfM-\bfQ_{h}|\|_{h,1}+\inf_{(v_{h},\bfzeta_{h})\in
W_h\times\Theta_h}|\|(w-v_{h},\bfbeta-\bfzeta_{h})|\|_{h,2}+ht\|\bfgamma\|_{1}+h\|\bfgamma\|_{0}.
\end{eqnarray}
A combination of \eqref{eq: 3} and \eqref{eq: 4} completes the proof.
\end{proof}

 To obtain the convergence order, we first need to consider error estimates for the approximations of finite element spaces in
 Lemma \ref{lemma: error MiSP3_1}-\ref{lemma: error MiSP3_2}.
\begin{lemma}\label{lemma: error MiSP3_1}
It holds
\begin{equation*}
 \inf\limits_{\bfQ_{h}\in\mathbb{M}_{h}}|\|\bfM-\bfQ_{h}|\|_{h,1}
 \lesssim
 h(\|\bfM\|_1+\|r\|_2+\|p\|_1+t\|p\|_2).
\end{equation*}
\end{lemma}
\begin{proof}
For the exact solution $\bfM$, first let $\bfQ_h^1$ be its piecewise constant $L^2$ projection, then
 $$\|\bfM-\bfQ_h^1\|_0\lesssim h\|\bfM\|_1.$$
Basing on Theorem \ref{Th:Regularity}, we have
$\bfgamma=\bfdiv\bfM=\grad r+\curl p, \text{~with~} (p,r)\in H_0^1(\Omega)\times \hat{H}^1(\Omega).$
Choose $\bfQ_h^2$ satisfying $\bfdiv_h\bfQ_h^2=\grad (I_h r)+\curlh(\Pi_h p)$
(we recall that $I_h$ and $\Pi_h$ are respectively the nodal interpolation and the  Scott-Zhang interpolation operators), and
$$\|\bfQ_h^2\|_0\approx h \|\bfdiv_h\bfQ_h^2\|_0\lesssim h(\|r\|_2+\|p\|_1).$$

Take $\bfQ_h=\bfQ_h^1+\bfQ_h^2$, then we can obtain the desired result
\begin{eqnarray*}
&&|\|\bfM-\bfQ_{h}|\|_{h,1}\\
&\le&\|\bfM-\bfQ_h^1\|_0+\|\bfQ_h^2\|_0+(h+t)\|\curl p-\curlh (\Pi_hp)\|_0\\
&&+\|\grad r-\grad(I_h r)\|_0+\|p-\Pi_hp\|_0\\
&\lesssim&h\|\bfM\|_1+h(\|r\|_2+\|p\|_1)+h\|p\|_1+ht\|p\|_2+h\|r\|_2+h\|p\|_1\\
&\lesssim&h\|\bfM\|_1+h\|r\|_2+h\|p\|_1+ht\|p\|_2,
\end{eqnarray*}
where  we have used the approximation properties
$$\|p-\Pi_hp\|_0\lesssim h\|p\|_1,\, \|\curlh(\Pi_hp)\|_0\lesssim \|p\|_1,\, \text{and~} \|\curl p-\curlh (\Pi_hp)\|_0\lesssim h\|p\|_2.$$
\end{proof}

\begin{lemma}\label{lemma: error MiSP3_2}
It holds
\begin{equation}
\inf\limits_{(v_{h},\bfzeta_{h})\in
W_{h}\times\Theta_{h}}|\|(w-v_{h},\bfbeta-\bfzeta_{h})|\|_{h,2}
\lesssim
 h(\|\bfbeta\|_{2}+\|w\|_{2}).
\end{equation}
\end{lemma}
\begin{proof}
By the definition of mesh-dependent norm, we immediately get
\begin{eqnarray*}
&&\inf\limits_{(v_{h},\bfzeta_{h})\in W_{h}\times\Theta_{h}}|\|(w-v_{h},\bfbeta-\bfzeta_{h})|\|_{h,2}\\
&=&\inf\limits_{\bfzeta_h\in \Theta_h}\|\bfepsilon(\bfbeta)-\bfepsilon(\bfzeta_h)\|_0
+\inf\limits_{v_{h}\in W_{h}}\|\bfR_h(\grad w)-\bfR_h(\grad v_h)\|_0 \\
&\le&\inf\limits_{\bfzeta_h\in \Theta_h}\|\bfepsilon(\bfbeta)-\bfepsilon(\bfzeta_h)\|_0
+\|\bfR_h(\grad w)-\grad w\|_0
+\inf\limits_{v_{h}\in W_{h}}\|\grad w-\grad v_h\|_0 \\
&\lesssim& h(\|\bfbeta\|_{2}+\|w\|_{2}).
\end{eqnarray*}
\end{proof}

 \begin{theorem}\label{theorem: error2-MiSP3} The discretization
 problem \eqref{eq:DiscreteForm2-a}-\eqref{eq:DiscreteForm2-b} admits a unique
 solution $(\bfM_{h},w_h, \bfbeta_{h})\in \mathbb{M}_{h}\times W_{h}\times\Theta_{h}$ such that
\begin{equation}\label{eq: ConvergenceOrder1-MiSP3}
|\|\bfM-\bfM_{h}|\|_{h,1}+|\|(w-w_{h},\bfbeta-\bfbeta_{h})|\|_{h,2}
 \lesssim h\left(\|\bfM\|_{1}+\|\bfbeta\|_{2}+\|w\|_2+\|r\|_2+\|p\|_1+t\|p\|_2\right)\lesssim h \|g\|_0.
\end{equation}
Furthermore, it holds
\begin{equation}\label{eq: ConvergenceOrder2-MiSP3}
 \begin{array}{ll}
 &\|\bfM-\bfM_{h}\|_{0}+(t+h)\|\bfgamma-\bfgamma_h\|_0+\|w-w_{h}\|_{1}+\|\bfbeta-\bfbeta_{h}\|_{1}\\
 \lesssim& h\left(\|\bfM\|_{1}+\|\bfbeta\|_{2}+\|w\|_2+\|r\|_2+\|p\|_1+t\|p\|_2\right)\lesssim h \|g\|_0.
 \end{array}
\end{equation}
\end{theorem}
\begin{proof}
\eqref{eq: ConvergenceOrder1-MiSP3} follows from Theorem \ref{theorem: error-MiSP3}, Lemma \ref{lemma: error MiSP3_1} and Lemma \ref{lemma: error MiSP3_2} directly.
 For \eqref{eq: ConvergenceOrder2-MiSP3}, basing on the definition of mesh-dependent norms, we only need to estimate $(t+h)\|\bfgamma-\bfgamma_h\|_0$ and $\|w-w_{h}\|_{1}$.

In fact, from the decomposition $\bfgamma=\grad r+\curlh p$ and $\bfgamma_h=\grad r_h+\curlh p_h$, 
we have
 \begin{eqnarray*}
 (t+h)\|\bfgamma-\bfgamma_h\|_0
 &=&(t+h)\|\grad (r-r_h)+\curlh (p-p_h)\|_0\\
 &\lesssim& \|\grad (r-r_h)\|_0+(t+h)\|\curlh (p-p_h)\|_0\\
 &\lesssim & |\|\bfM-\bfM_h|\|_{h,1}.
\end{eqnarray*}

And the error estimate for $\|w-w_{h}\|_{1}$ can be obtained from the triangle inequality:
 \begin{eqnarray*}
 \|\grad w-\grad w_{h}\|_{0}
 &=&\|\grad w-\bfR_h\grad w+\bfR_h(\grad w-\grad w_{h})\|_0\\
 &\le& \|\grad w-\bfR_h\grad w\|_0+|\|(w-w_{h},\bfbeta-\bfbeta_{h})|\|_{h,2}\\
 &\lesssim &h\|w\|_2+|\|(w-w_{h},\bfbeta-\bfbeta_{h})|\|_{h,2}.
\end{eqnarray*}
Then an application of \eqref{eq: ConvergenceOrder1-MiSP3} implies \eqref{eq: ConvergenceOrder2-MiSP3}.
\end{proof}

\section{error analysis for MiSP4}
 This section is denoted to the error estimates for the MiSP4 element. The corresponding subspaces in this section
 are defined as in subsection 3.2. The error analysis for MiSP4 is similar as for MiSP3. And first we also give the following properties for the operator $\bfR_h$.
 \begin{lemma}\label{lemma: error_Rh_MiSP4}\cite[Lemma 2.1]{Duran-Hernandez-Nieto-Liberman-Rodriguez2004}
 $\bfR_h(\grad v_{h})=\grad v_{h}, \forall v_h\in W_h$.
 \end{lemma}
 \begin{lemma} \cite[Theorem \textrm{III} 3.4]{Girault-Raviart}
 $\|\bfeta-\bfR_h\bfeta\|_{0}\lesssim h\|\bfeta\|_{1}, \forall \bfeta\in H^{1}(\Omega)^{2}\bigcap
 H_{0}(rot,\Omega)$.
 \end{lemma}

 We introduce two mesh-dependent norms for the finite dimensional
spaces:

For any $\bfQ\in (H^1(\Omega))^{2\times 2}_{sym}\bigcup\mathbb{M}_h$, $v\in (H^2(\Omega)\bigcap H_0^1(\Omega))\bigcup W_h$, $\bfzeta\in H_0^1(\Omega)^2\bigcup \Theta_h$, define
 \begin{equation}
 |\|\bfQ|\|_{h,1}:=\|\bfQ\|_{0}+(t+h)\|\bfdiv_h\bfQ\|_{0},
 \end{equation}
 \begin{equation}
 |\|(v,\bfzeta)|\|_{h,2}:=\|\bfepsilon(\bfzeta)\|_{0}+(t+h)^{-1}\|\bfR_h(\grad
 v-\bfzeta)\|_{0}.
 \end{equation}

 With the definition of mesh-dependent norms, it is easy to check the continuity results in Lemma \ref{lemma: Continuity-MiSP4}.
 While the corresponding coercivity results are deduced in Lemma \ref{lemma: Coercivity1-MiSP4}-\ref{lemma: Coercivity2-MiSP4}.
 Lemma \ref{Lemma: PrepareWork} is a preparation for Lemma \ref{lemma: Coercivity2-MiSP4}.
\begin{lemma}\label{lemma: Continuity-MiSP4}
For any $\bfM,\bfQ\in (H^1(\Omega))^{2\times 2}_{sym}\bigcup\mathbb{M}_h$, $v\in
(H^2(\Omega)\bigcap H_0^1(\Omega))\bigcup W_h$, $\bfzeta\in H_0^1(\Omega)^2\bigcup \Theta_h$, it holds uniformly the continuity conditions
\begin{equation}
a(\bfM,\bfdiv_h\bfM;\bfQ,\bfdiv_h\bfQ)\lesssim |\|\bfM|\|_{h,1}|\|\bfQ|\|_{h,1},
\end{equation}
 \begin{equation}
 \tilde{b}(\bfQ,\bfdiv_h\bfQ;v,\bfzeta) \lesssim  |\|\bfQ|\|_{h,1}|\|(v,\bfzeta)|\|_{h,2}.
 \end{equation}
\end{lemma}

\begin{lemma}\label{lemma: Coercivity1-MiSP4}
 It holds uniformly the discrete coercivity condition
 \begin{equation}
 a(\bfQ_{h},\bfdiv_h\bfQ_h;\bfQ_{h},\bfdiv_h\bfQ_h)\gtrsim
 |\|\bfQ_h|\|_{h,1}\quad \mbox{~~for all }
 \bfQ_h\in \mathbb{M}_h.
 \end{equation}
 \end{lemma}
\begin{proof}
The proof immediately follows from the inverse inequality
$\|\bfdiv_{h}\bfQ_{h}\|_{0}\,\leq|\bfQ_{h}|_{1}\lesssim
h^{-1}\|\bfQ_{h}\|_{0}.$
\end{proof}

\begin{lemma}\label{Lemma: PrepareWork}
The following two conclusions hold:\\
 (1) For any given $\bfzeta_h\in\Theta_h$, there exists $\bfQ_h^1\in\mathbb{M}_h$, such that
\begin{equation}\label{inf-sup1}
(\bfQ_h^1 ,\bfepsilon(\bfzeta_{h}))=\|\bfQ_h^1\|_0^2\thickapprox\|\bfepsilon(\bfzeta_{h})\|_0^2,
\mbox{~~and~}\bfdiv_h\bfQ_h^1=0;
\end{equation}
(2)For any given $v_h\in W_h$, $\bfzeta_h\in\Theta_h$, there exists
$\bfQ_h^2\in\mathbb{M}_h$,
such that
\begin{equation}\label{inf-sup2}
(\bfdiv_h \bfQ_h^2, \bfR_h(\grad
v_{h}-\bfzeta_{h}))=-(t^2+h^2)\|\bfdiv_h \bfQ_h^2  \|_{0}^2\thickapprox-\frac1{t^2+h^2}\|\bfR_h(\grad
v_{h}-\bfzeta_{h})\|_{0}^2,
\end{equation}
and
\begin{equation}\label{equivalence}
\|\bfdiv_h \bfQ_h^2  \|_{0}\thickapprox
h^{-1}\|\bfQ_h^2  \|_{0}.
\end{equation}
\end{lemma}
\begin{proof}
The proof is similar to that in \cite{C;Xie;Yu;Zhou2011}.

(1) Given $\bfzeta_h\in\Theta_h$, choose $\bfQ_h^1 $ as the 5-parameter PS element in \cite{C;Xie;Yu;Zhou2011}. The proof for \eqref{inf-sup1} can be found in \cite[Lemma 4.4]{C;Xie;Yu;Zhou2011}.

(2) Given $v_h\in W_h$, $\bfzeta_h\in\Theta_h$, for any $K\in\mathcal{T}_h$, $\bfR_h(\grad v_{h}-\bfzeta_{h})|_K$ can be expressed as
 \begin{eqnarray*}
 &&\bfR_h(\grad v_{h}-\bfzeta_{h})|_K\\
 &=&\frac1{J_K}
 \left(\begin{array}{cc}
 b_2+b_{12}\xi &-(b_1+b_{12}\eta)\\
 -(a_2+a_{12}\xi) &a_1+a_{12}\eta
 \end{array}\right)
 \left(\begin{array}{cccc}
 1 &\eta &0 &0\\
 0 &0 &1 &\xi
 \end{array}\right)
  \left(\begin{array}{c}
 c_1\\
 c_2\\
 c_3\\
 c_4
 \end{array}\right),  \mbox{here} \left(\begin{array}{c}
 c_1\\
 c_2\\
 c_3\\
 c_4
 \end{array}\right) \mbox{depends on~} v_h,\bfzeta_h.
 \end{eqnarray*}
Some calculations show
 \begin{eqnarray*}
 &&\|\bfR_h(\grad v_{h}-\bfzeta_{h})\|_{0,K}^2\\
 &=&\frac{4}{J_K(\xi_1,\eta_1)}
 \left[(b_2c_1-b_1c_3)^2+\frac13(b_{2}c_2-b_{12}c_3)^2+\frac13(b_{12}c_1-b_1c_4)^2+\frac19(b_{12}c_2-b_{12}c_4)^2\right.\\
 &&\hspace{0.7in}  \left.+(a_2c_1-a_1c_3)^2+\frac13(a_{2}c_2-a_{12}c_3)^2+\frac13(a_{12}c_1-a_1c_4)^2+\frac19(a_{12}c_2-a_{12}c_4)^2\right]\\
 &=&\frac{C_1}{J_K(\xi_1,\eta_1)}\left[(b_2c_1-b_1c_3)^2+(b_{2}c_2-b_{12}c_3)^2+(a_2c_1-a_1c_3)^2+(a_{12}c_1-a_{1}c_4)^2\right].
 \end{eqnarray*}

 Take $\bfQ_{h}|_K=\left(\begin{array}{c}
 c_1\xi+c_3\eta+c_2\xi\eta\\
 c_1\xi+c_3\eta+c_2\xi\eta\\
 0\end{array}\right)$,
 then we have
 $$\bfdiv_{h}\bfQ_{h}|_K
 =\frac{1}{J_K}\left(\begin{array}{c}
 (b_2c_1-b_1c_3)+(b_{12}c_1-b_1c_2)\xi+(b_{2}c_2-b_{12}c_3)\eta\\
 -(a_2c_1-a_1c_3)-(a_{12}c_1-a_1c_4)\xi+(a_{2}c_4-a_{12}c_3)\eta
 \end{array}\right)$$
 and
  \begin{eqnarray*}
 \|\bfdiv_{h}\bfQ_{h}\|_{0,K}^2
 &=&\frac{4}{J_K(\xi_2,\eta_2)}
 \left[(b_2c_1-b_1c_3)^2+\frac13(b_{2}c_2-b_{12}c_3)^2+\frac13(b_{12}c_1-b_1c_4)^2\right.\\
 &&\hspace{0.7in}  \left.+(a_2c_1-a_1c_3)^2+\frac13(a_{2}c_4-a_{12}c_3)^2+\frac13(a_{12}c_1-a_1c_4)^2\right]\\
 &=&\frac{C_2}{J_K(\xi_2,\eta_2)}\left[(b_2c_1-b_1c_3)^2+(b_{2}c_2-b_{12}c_3)^2+(a_2c_1-a_1c_3)^2+(a_{12}c_1-a_{1}c_4)^2\right].
 \end{eqnarray*}
On the other hand, it holds
 \begin{eqnarray*}
 &&\int_{K}\bfdiv_{h}\bfQ_{h}\cdot\bfR_h(\grad v_{h}-\bfzeta_{h})dxdy\\
 &=&\frac{4}{J_K(\xi_3,\eta_3)}
 \left[(b_2c_1-b_1c_3)^2+\frac13(b_{2}c_2-b_{12}c_3)^2+\frac13(b_{12}c_1-b_1c_4)^2\right.\\
 &&\hspace{0.7in}  \left.+(a_2c_1-a_1c_3)^2+\frac13(a_{2}c_4-a_{12}c_3)^2+\frac13(a_{12}c_1-a_1c_4)^2\right]\\
 &=&\frac{C_3}{J_K(\xi_3,\eta_3)}\left[(b_2c_1-b_1c_3)^2+(b_{2}c_2-b_{12}c_3)^2+(a_2c_1-a_1c_3)^2+(a_{12}c_1-a_{1}c_4)^2\right].
 \end{eqnarray*}

Let $C_0=-\frac{C_3}{C_2}\frac{J_K(\xi_2,\eta_2)}{J_K(\xi_3,\eta_3)}\frac1{t^2+h^2}$, and choose $\bfQ_h^2|_K=C_0\bfQ_h|_K$, i.e. $\bfdiv_{h}\bfQ_h^2 |_K =C_0\bfdiv_{h}\bfQ_{h}|_K$, then a summation over all elements in $\mathcal{T}_h$ completes the proof for \eqref{inf-sup2}. The result \eqref{equivalence} follows from
 the construction of $\bfQ_h^2$.
\end{proof}

\begin{lemma}\label{lemma: Coercivity2-MiSP4}
It holds the inf-sup condition
\begin{equation}\label{inf-sup}
\sup_{\bfQ_{h}\in\mathbb{M}_{h}}
\frac{\tilde{b}(\bfQ_{h},\bfdiv_h\bfQ_h;v_{h},\bfzeta_{h})}{|\|\bfQ_{h}|\|_{h,1}}\gtrsim|\|(v_{h},\bfzeta_{h})|\|_{h,2},
\mbox{ for all } (v_{h},\bfzeta_{h})\in W_{h}\times \Theta_{h}.
\end{equation}
\end{lemma}

\begin{proof}
 For $\bfzeta_h\in\Theta_h$, from \eqref{inf-sup1} there exists a positive constant
 $C_1$ and $\bfQ_h^1 \in\mathbb{M}_h$, such that
 \begin{equation}\label{eq:satisfy1}
 (\bfQ_h^1 ,\bfepsilon(\bfzeta_{h}))=\|\bfQ_h^1 \|_0^2=C_1\|\bfepsilon(\bfzeta_{h})\|_0^2,
 \mbox{~~and~}\bfdiv_h\bfQ_h^1 =0.
 \end{equation}
 For $v_h\in W_h$, $\bfzeta_h\in\Theta_h$, from \eqref{inf-sup2}
 for any positive constant $C_2$ there exists $\bfQ_h^2  \in\mathbb{M}_h$,
such that
\begin{equation}\label{eq:satisfy2}
 (\bfdiv_h\bfQ_h^2  , \bfR_h(\grad v_{h}-\bfzeta_{h}))
 =-C_2(t^2+h^2)\|\bfdiv_h\bfQ_h^2  \|_{0}^2
 =-C_2^{-1}(t^2+h^2)^{-1}\|\bfR_h(\grad
 v_{h}-\bfzeta_{h})\|_{0}^2,
\end{equation}
 and there exists a positive constant $C_3$ independent of $h$ and
 $t$, such that
 \begin{equation}
 \|\bfdiv_h\bfQ_h^2  \|_{0}^2=C_3h^{-2}\|\bfQ_h^2  \|_{0}^2.
\end{equation}

Let $\bfQ_h=\bfQ_h^1 +\bfQ_h^2  $, then we have
\begin{eqnarray*}
&&\tilde{b}(\bfQ_h,\bfdiv_h\bfQ_h;v_{h},\bfzeta_{h})\\
&=&(\bfQ_h^1 +\bfQ_h^2  ,\bfepsilon(\bfzeta_{h}))-(\bfdiv_h\bfQ_h^1 +\bfdiv_h\bfQ_h^2  ,\bfR_h(\grad
v_{h}-\bfzeta_{h}))\\
&=&(\bfQ_h^1 ,\bfepsilon(\bfzeta_{h}))+(\bfQ_h^2  ,\bfepsilon(\bfzeta_{h}))-(\bfdiv_h\bfQ_h^2  ,\bfR_h(\grad
v_{h}-\bfzeta_{h}))\\
&\geq
&\|\bfQ_h^1 \|_{0}^{2}-\|\bfQ_h^2  \|_{0}\|\bfepsilon(\bfzeta_{h})\|_{0}+C_2(t^{2}+h^{2})\|\bfdiv_h\bfQ_h^2  \|_{0}^{2}\\
&\geq
&\|\bfQ_h^1 \|_{0}^{2}-\frac{C_{1}}{2}\|\bfepsilon(\bfzeta_{h})\|_{0}^2-\frac{1}{2C_1}\|\bfQ_h^2  \|_{0}^{2}
 +C_{2}(t^{2}+h^2)\|\bfdiv_h\bfQ_h^2  \|_{0}^{2}\\
 &\geq&\|\bfQ_h^1 \|_{0}^{2}-\frac{C_{1}}{2}\|\bfepsilon(\bfzeta_{h})\|_{0}^2-\frac{h^2}{2C_1C_3}\|\bfdiv_h\bfQ_h^2  \|_{0}^{2}
 +C_{2}(t^{2}+h^2)\|\bfdiv_h\bfQ_h^2  \|_{0}^{2}\\
 &\geq
&\frac{C_{1}}{2}\|\bfQ_h^1 \|_{0}^{2}
 +\frac{C_2}{2}(t^2+h^2)\|\bfdiv_h\bfQ_h^2  \|_{0}^{2} \mbox{~~~(by taking $C_2=\frac{1}{C_1C_3}$)}\\
&\approx
&\|\bfepsilon(\bfzeta_h)\|_{0}^{2}+(t^2+h^2)^{-1}\|\bfR_h(\grad
 v_{h}-\bfzeta_{h})\|_{0}^2\\
&\approx
&\|\bfQ_h^1 +\bfQ_h^2  \|_{0}^{2}+(t^2+h^2)\|\bfdiv_h\bfQ_h^1 +\bfdiv_h\bfQ_h^2 \|_0^2
=\|\bfQ_h\|_{0}^{2}+(t^2+h^2)\|\bfdiv_h\bfQ_h\|_0^2.
\end{eqnarray*}
 This immediately indicates
 \begin{eqnarray*}
\sup_{\bfQ_{h}\in\mathbb{M}_{h}}
\frac{\tilde{b}(\bfQ_{h},\bfdiv_h\bfQ_h;v_{h},\bfzeta_{h})}{|\|\bfQ_{h}|\|_{h,1}}
\gtrsim|\|(v_h,\bfzeta_h)|\|_{h,2}.
\end{eqnarray*}
\end{proof}

With the above continuity and coercivity results, we can obtain the following error estimates for MiSP4 element by following the same way as in Theorem \ref{theorem: error-MiSP3}.
\begin{theorem}\label{theorem: error-MiSP4}
 Let $(\bfM,\bfgamma=\bfdiv_h\bfM_h,w,\bfbeta)\in \mathbb{M}\times\Gamma\times W\times \Theta$ be the
 solution of the problem \eqref{eq:WeakForm2-a}-\eqref{eq:WeakForm2-b}. Then the discretization
 problem \eqref{eq:DiscreteForm2-a}-\eqref{eq:DiscreteForm2-b} admits a unique solution
 $(\bfM_{h},w_{h},\bfbeta_{h})\in \mathbb{M}_{h}\times W_{h}\times\Theta_{h}$ such that
\begin{eqnarray*}
&&|\|\bfM-\bfM_{h}|\|_{h,1}+|\|(w-w_{h},\bfbeta-\bfbeta_{h})|\|_{h,2}\\
& \lesssim &\inf\limits_{\bfQ_{h}\in
\mathbb{M}_{h}}|\|\bfM-\bfQ_{h}|\|_{h,1}+\inf\limits_{(v_{h},\bfzeta_{h})\in W_h\times\Theta_h}|\|(w-v_{h},\bfbeta-\bfzeta_{h})|\|_{h,2}+ht\|\bfgamma\|_{1}+h\|\bfgamma\|_{0}.\\
\end{eqnarray*}
\end{theorem}

Next we consider the approximation properties of finite element spaces. Lemma \ref{lemma: error_MiSP4_1}
gives the error estimates for space $\mathbb{\bfM}_h$, and Lemma \ref{lemma: error_MiSP4_2} is for space $W_h\times\Theta_h$.
We need to notice here the key for Lemma \ref{lemma: error_MiSP4_2} is the property of the operator $\bfR_h$ described in Lemma \ref{lemma projection}.
Finally the  convergence theorem, i.e. Theorem \ref{theorem: error2-MiSP4}, follows from these lemmas.

\begin{lemma}\label{lemma: error_MiSP4_1}
It holds
\begin{equation*}
 \inf\limits_{\bfQ_{h}\in
 \mathbb{M}_{h}}|\|\bfM-\bfQ_{h}|\|_{h,1}
 \lesssim
 h\left(\|\bfM\|_1+\|\bfgamma\|_0+t\|\bfgamma\|_1\right).
\end{equation*}
\end{lemma}
\begin{proof}
For the exact solution $\bfM$, first let $\bfQ_h^1$ be its piecewise constant $L^2$ projection, then
 $$\|\bfM-\bfQ_h^1\|_0\lesssim h\|\bfM\|_1.$$
For the exact solution $\bfgamma$, secondly choose $\bfQ_h^2$ satisfying:

(1) $\bfdiv_h\bfQ_h^2$ is the piecewise constant $L^2$ projection of $\bfgamma$,
then
$$\|\bfgamma-\bfdiv_h\bfQ_h^2\|_0\approx h \|\bfgamma\|_1, \quad  \|\bfdiv_h\bfQ_h^2\|_0\lesssim\|\bfgamma\|_0;$$

(2) $\|\bfQ_h^2\|_0\approx h \|\bfdiv_h\bfQ_h^2\|_0$, then  $\|\bfQ_h^2\|_0\lesssim h\|\bfgamma\|_0.$

Take $\bfQ_h=\bfQ_h^1+\bfQ_h^2$, then we get the desired result
\begin{eqnarray*}
&&|\|\bfM-\bfQ_{h}|\|_{h,1}\\
&\le&\|\bfM-\bfQ_h^1\|_0+\|\bfQ_h^2\|_0+(h+t)\|\bfgamma-\bfdiv_h\bfQ_{h}^2\|_0\\
&\lesssim&h\|\bfM\|_1+h\|\bfgamma\|_0+h\|\bfgamma-\bfdiv_h\bfQ_{h}^2\|_0+t\|\bfgamma-\bfdiv_h\bfQ_{h}^2\|_0\\
&\lesssim&h\|\bfM\|_1+h\|\bfgamma\|_0+h\|\bfgamma\|_0+th\|\bfgamma\|_1\\
&\lesssim&h\left(\|\bfM\|_1+\|\bfgamma\|_0+t\|\bfgamma\|_1\right).
\end{eqnarray*}
\end{proof}

\begin{remark}
We note that with the same technique as in Lemma \ref{lemma: error_MiSP4_1},  the condition $t\lesssim h$  in \cite[Lemma 3.2]{C;Xie;Yu;Zhou2011} and in \cite[Theorem 4.3]{C;Xie;Yu;Zhou2011}  can be removed.
\end{remark}

\begin{assumption}\label{assumption} \cite{Duran-Hernandez-Nieto-Liberman-Rodriguez2004}
The mesh $\mathcal{T}_{h}$ is a refinement of a coarser partition
$\mathcal{T}_{2h}$, obtained by jointing the midpoints of each
opposite edge in each $K_{2h}\in \mathcal{T}_{2h}$ (called
macroelement). In addition, $\mathcal{T}_{2h}$ is a similar
refinement of a still coarser regular partition $\mathcal{T}_{4h}$.
\end{assumption}

\begin{lemma}\label{lemma projection}
 \cite[Lemma 3.2, 3.4]{Duran-Hernandez-Nieto-Liberman-Rodriguez2004}Under Assumption \ref{assumption}, let $W_{h}$, $\Theta_{h}$, $Z_{h}$ and the operator $\bfR_h$ be
defined as before. Then for the given $(w,\bfbeta)$, there exist
$\hat{w}\in W_{h}$ and  $\hat{\bfbeta}\in \Theta_{h}$ and operator
$\boldsymbol{\Pi}: H_{0}(rot,\Omega)\bigcap
H^{1}(\Omega)^{2}\rightarrow Z_{h}$ satisfying
\begin{equation}
\|\bfbeta-\hat{\bfbeta}\|_{1}\lesssim h\|\bfbeta\|_{2},
\end{equation}
\begin{equation}
 \bfR_h(\grad \hat{w}-\hat{\bfbeta})=\boldsymbol{\Pi}(\grad w-\bfbeta),
\end{equation}
and
\begin{equation}
 \|\boldsymbol{\eta}-\boldsymbol{\Pi}\boldsymbol{\eta}\|_{0}\lesssim h\|\boldsymbol{\eta}\|_{1}, \forall \boldsymbol{\eta}\in H_{0}(rot,\Omega)\bigcap
 H^{1}(\Omega)^{2}.
\end{equation}
\end{lemma}

\begin{lemma}\label{lemma: error_MiSP4_2}
Under Assumption \ref{assumption}, it holds
\begin{equation}
\inf\limits_{(v_{h},\bfzeta_{h})\in
W_{h}\times\Theta_{h}}|\|(w-v_{h},\bfbeta-\bfzeta_{h})|\|_{h,2}
\lesssim
 h\|\bfbeta\|_{2}+\frac{ht^{2}}{t+h}\|\bfgamma\|_{1}.
\end{equation}
\end{lemma}
\begin{proof}
 Choose $(v_{h},\bfzeta_{h})=(\hat{w},\hat{\bfbeta})$, with $(\hat{w},\hat{\bfbeta})\in W_{h}\times\Theta_{h}$ as in Lemma \ref{lemma
 projection}, then we can get
\begin{equation*}
\begin{array}{ll}
&\inf\limits_{(v_{h},\bfzeta_{h})\in W_{h}\times\Theta_{h}}|\|(w-v_{h},\bfbeta-\bfzeta_{h})|\|_{h,2}\\
 \displaystyle=&\inf\limits_{(v_{h},\bfzeta_{h})\in W_{h}\times\Theta_{h}}
\|\bfepsilon(\bfbeta)-\bfepsilon(\bfzeta_{h})\|_0+\frac{1}{t+h}\|\bfR_h(\grad
w-\bfbeta)-\bfR_h(\grad v_{h}-\bfzeta_{h})\|_{0}\\
 \displaystyle\leq
&\|\bfepsilon(\bfbeta)-\bfepsilon(\hat{\bfbeta})\|_0+\frac{1}{t+h}\|\bfR_h(\grad
w-\bfbeta)-\bfR_h(\grad \hat{w}-\hat{\bfbeta})\|_{0}\\
 \displaystyle=&\|\bfepsilon(\bfbeta)-\bfepsilon(\hat{\bfbeta})\|_0+\frac{1}{t+h}\|\bfR_h(\grad
w-\bfbeta)-\boldsymbol{\Pi}(\grad w-\bfbeta)\|_{0}\\
 \displaystyle \leq
&\|\bfepsilon(\bfbeta)-\bfepsilon(\hat{\bfbeta})\|_0+\frac{1}{t+h}\|\bfR_h(\grad
w-\bfbeta)-(\grad w-\bfbeta)\|_{0}+\frac{1}{t+h}\|(\grad
w-\bfbeta)-\boldsymbol{\Pi}(\grad w-\bfbeta)\|_{0}\\
 \displaystyle \lesssim &h\|\bfbeta\|_{2}+\frac{ht^{2}}{t+h}\|\lambda t^{-2}(\grad
w-\bfbeta)\|_{1}\\
 \displaystyle\lesssim &h\|\bfbeta\|_{2}+\frac{ht^{2}}{t+h}\|\bfgamma\|_{1}.
\end{array}
\end{equation*}
\end{proof}

 \begin{theorem}\label{theorem: error2-MiSP4} Under Assumption \ref{assumption}, the discretization
 problem \eqref{eq:DiscreteForm2-a}-\eqref{eq:DiscreteForm2-b} admits a unique
 solution $(\bfM_{h},w_{h},\bfbeta_{h})\in \mathbb{M}_{h}\times\Gamma_h\times
 W_{h}\times\Theta_{h}$ such that
\begin{equation}\label{eq: ConvergenceOrder2-MiSP4}
|\|\bfM-\bfM_{h}|\|_{h,1}+|\|(w-w_{h},\bfbeta-\bfbeta_{h})|\|_{h,2}
 \lesssim h(\|\bfM\|_{1}+\|\bfbeta\|_{2}+\|\bfgamma\|_{0}+t\|\bfgamma\|_{1})\lesssim h \|g\|_0.
\end{equation}
Furthermore, it holds
\begin{equation}
 \begin{array}{ll}
 &\|\bfM-\bfM_{h}\|_{0}+(t+h)\|\bfgamma-\bfgamma_h\|_0+\|w-w_{h}\|_{1}+\|\bfbeta-\bfbeta_{h}\|_{1}\\
 \lesssim& h(\|\bfM\|_{1}+\|w\|_{2}+\|\bfbeta\|_{2}+t\|\bfgamma\|_{1}+\|\bfgamma\|_{0})\lesssim h \|g\|_0.
 \end{array}
\end{equation}
\end{theorem}
\begin{proof}
 The estimate \eqref{eq: ConvergenceOrder2-MiSP4} follows from the Theorem \ref{theorem: error-MiSP4}, Lemma \ref{lemma: error_MiSP4_1} and Lemma \ref{lemma: error_MiSP4_2}.

  For the second estimate, we only need to estimate $\|w-w_{h}\|_{1}$.
 In fact,
 \begin{eqnarray*}
 &&\|\grad w-\grad w_{h}\|_{0}\\
 &=&\|\grad w-\bfR_h\grad w+\bfR_h(\grad w-\grad
 w_{h}-\bfbeta+\bfbeta_{h})+\bfR_h(\bfbeta-\bfbeta_{h})\|_0\\
 &\leq&\|\grad w-\bfR_h\grad w\|_{0}+\|\bfR_h(\grad w-\grad
 w_{h}-\bfbeta+\bfbeta_{h})\|_{0}+\|\bfR_h(\bfbeta-\bfbeta_{h})\|_0\\
&\lesssim
&h(\|w\|_2+\|\bfM\|_1+\|\bfbeta\|_2+h\|\bfgamma\|_1+\|\bfgamma\|_0).
\end{eqnarray*}
\end{proof}

\section{Numerical Results}
We compute a square plate with analytical solution to show the convergence.
This example is taken from \cite{Hu.J;Shi.Z2008}. The domain is the unit square $(0,1)^2$, the material parameters
are taken as $E=1.0$, $\nu=0.3$ and $\kappa=\frac56$. The exact solution
is: the first component of the rotation $\beta_1=100y^3(y-1)^3x^2(x-1)^(2x-1)$,
 the second component of the rotation $\beta_2=100x^3(x-1)^3y^2(y-1)^(2y-1)$,
 and the displacement $w=100(\frac13x^3(x-1)^3y^3(y-1)^3-\frac{2t^2}{5(1-\nu)}[y^3(y-1)^3x(x-1)(5x^2-5x+1)+x^3(x-1)^3y(y-1)(5y^2-5y+1)])$.
 Therefore, the transverse load
 $g=\frac{200E}{1-\nu^2}(x^3(x-1)^3(5y^2-5y+1)+y^3(y-1)^3(5x^2-5x+1)+x(x-1)y(y-1)(5x^2-5x+1)(5y^2-5y+1))$.
 For the plate thickness $t$, we consider four cases: $t=1.0, 0.1, 0.001, 1e-8$.

 The results for MiSP3 method under the uniform meshes (Figure \ref{fig:MiSP3-RegularMesh}) are reported in Table \ref{table:MiSP3-SquarePlate-ConvergenceOrder},
 while the results for MiSP4 method under the uniform meshes (Figure \ref{fig:MiSP4-RegularMesh}) are reported in Table \ref{table:MiSP4-SquarePlate-ConvergenceOrder}.
 These results are conformable to the error estimates in Theorem \ref{theorem: error2-MiSP3} and Theorem \ref{theorem: error2-MiSP4}.

\begin{figure}[!h]
\begin{center}
\setlength{\unitlength}{0.36cm}
\begin{picture}(28,9)
\put(0,0){\line(1,0){8}}
\put(0,4){\line(1,0){8}}
\put(0,8){\line(1,0){8}}
\put(0,0){\line(0,1){8}}
\put(4,0){\line(0,1){8}}
\put(8,0){\line(0,1){8}}

\put(0,0){\line(1,1){8}}
\put(0,4){\line(1,1){4}}
\put(4,0){\line(1,1){4}}

\put(10,0){\line(1,0){8}}
\put(10,2){\line(1,0){8}}
\put(10,4){\line(1,0){8}}
\put(10,6){\line(1,0){8}}
\put(10,8){\line(1,0){8}}
\put(10,0){\line(0,1){8}}
\put(12,0){\line(0,1){8}}
\put(14,0){\line(0,1){8}}
\put(16,0){\line(0,1){8}}
\put(18,0){\line(0,1){8}}

\put(10,0){\line(1,1){8}}
\put(10,2){\line(1,1){6}}
\put(10,4){\line(1,1){4}}
\put(10,6){\line(1,1){2}}
\put(12,0){\line(1,1){6}}
\put(14,0){\line(1,1){4}}
\put(16,0){\line(1,1){2}}

\put(20,0){\line(1,0){8}}
\put(20,1){\line(1,0){8}}
\put(20,2){\line(1,0){8}}
\put(20,3){\line(1,0){8}}
\put(20,4){\line(1,0){8}}
\put(20,5){\line(1,0){8}}
\put(20,6){\line(1,0){8}}
\put(20,7){\line(1,0){8}}
\put(20,8){\line(1,0){8}}
\put(20,0){\line(0,1){8}}
\put(21,0){\line(0,1){8}}
\put(22,0){\line(0,1){8}}
\put(23,0){\line(0,1){8}}
\put(24,0){\line(0,1){8}}
\put(25,0){\line(0,1){8}}
\put(26,0){\line(0,1){8}}
\put(27,0){\line(0,1){8}}
\put(28,0){\line(0,1){8}}

\put(20,0){\line(1,1){8}}
\put(20,1){\line(1,1){7}}
\put(20,2){\line(1,1){6}}
\put(20,3){\line(1,1){5}}
\put(20,4){\line(1,1){4}}
\put(20,5){\line(1,1){3}}
\put(20,6){\line(1,1){2}}
\put(20,7){\line(1,1){1}}
\put(21,0){\line(1,1){7}}
\put(22,0){\line(1,1){6}}
\put(23,0){\line(1,1){5}}
\put(24,0){\line(1,1){4}}
\put(25,0){\line(1,1){3}}
\put(26,0){\line(1,1){2}}
\put(27,0){\line(1,1){1}}
\end{picture}
\end{center}
\caption{Uniform mesh\label{fig:MiSP3-RegularMesh}}
\end{figure}
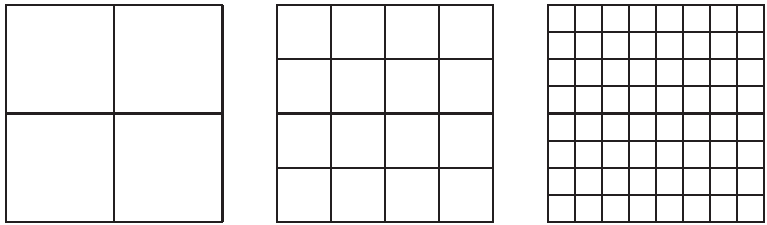

\begin{table}[!h]\renewcommand{\baselinestretch}{1.10}\small
\begin{center}
 \caption{Results of error on uniform mesh with MiSP3\label{table:MiSP3-SquarePlate-ConvergenceOrder}}
\begin{tabular}{cccccccc}\hline
$t$ &                         &$4\times 4$ &$8\times 8$ &$16\times 16$ &$32\times 32$ &$64\times 64$ &rate\\\hline
1   &$|w-w_h|_1$                  &0.2834    &0.1679    &0.0877    &0.0443    &0.0222    &0.9182\\
    &$|\beta-\beta_h|_1$          &0.0820    &0.0461    &0.0238    &0.0120    &0.0060    &0.9427\\
    &$\|M-M_h\|_0$                &0.0070    &0.0033    &0.0015    &0.0008    &0.0004    &1.0543\\
    &$\|\gamma-\gamma_h\|_0$      & 0.0882   &0.0525    &0.0275    &0.0139    &0.0070    &0.9156\\
    &$(t+h)\|\gamma-\gamma_h\|_0$ &0.1194    &0.0618    &0.0299    &0.0145    &0.0071    &1.0169\\ \hline
0.1 &$|w-w_h|_1$                  &0.0132    &0.0066    &0.0032    &0.0016    &0.0008    &1.0153\\
    &$|\beta-\beta_h|_1$          &0.0824    &0.0460    &0.0238    &0.0120    &0.0060    &0.9445\\
    &$\|M-M_h\|_0$                &0.0069    &0.0032    &0.0015    &0.0008    &0.0004    &1.0520\\
    &$\|\gamma-\gamma_h\|_0$      &0.0851    &0.0501    &0.0270    &0.0138    &0.0070    &0.9031\\
    &$(t+h)\|\gamma-\gamma_h\|_0$ &0.0386    &0.0139    &0.0051    &0.0020    &0.0008    &1.3764\\  \hline
0.001 &$|w-w_h|_1$                &0.0112    &0.0053    &0.0025    &0.0012    &0.0006    &1.0520\\
     &$|\beta-\beta_h|_1$          &0.0838    &0.0463    &0.0238    &0.0120    &0.0060    &0.9506\\
     &$\|M-M_h\|_0$                &0.0070    &0.0033    &0.0016    &0.0008    &0.0004    &1.0569\\
     &$\|\gamma-\gamma_h\|_0$      &0.0840    &0.0496    &0.0294    &0.0166    &0.0094    &0.7902\\
     &$(t+h)\|\gamma-\gamma_h\|_0$ &0.0298    &0.0088    &0.0026    &0.0007    &0.0002    &1.7753\\  \hline
1e-8 &$|w-w_h|_1$                &0.0112    &0.0053    &0.0025    &0.0012    &0.0006    &1.0520\\
     &$|\beta-\beta_h|_1$        &0.0838    &0.0463    &0.0238    &0.0120    &0.0060    &0.9506\\
     &$\|M-M_h\|_0$              &0.0070    &0.0033    &0.0016    &0.0008    &0.0004    &1.0569\\
     &$\|\gamma-\gamma_h\|_0$    &0.0840    &0.0497    &0.0294    &0.0167    &0.0097    &0.7781\\
     &$(t+h)\|\gamma-\gamma_h\|_0$ &0.0297  &0.0088    &0.0026    &0.0007    &0.0002    &1.7781\\\hline
\end{tabular}
\end{center}
\end{table}

 \begin{figure}[!h]
\begin{center}
\setlength{\unitlength}{0.36cm}
\begin{picture}(28,9)
\put(0,0){\line(1,0){8}}
\put(0,4){\line(1,0){8}}
\put(0,8){\line(1,0){8}}
\put(0,0){\line(0,1){8}}
\put(4,0){\line(0,1){8}}
\put(8,0){\line(0,1){8}}

\put(10,0){\line(1,0){8}}
\put(10,2){\line(1,0){8}}
\put(10,4){\line(1,0){8}}
\put(10,6){\line(1,0){8}}
\put(10,8){\line(1,0){8}}
\put(10,0){\line(0,1){8}}
\put(12,0){\line(0,1){8}}
\put(14,0){\line(0,1){8}}
\put(16,0){\line(0,1){8}}
\put(18,0){\line(0,1){8}}

\put(20,0){\line(1,0){8}}
\put(20,1){\line(1,0){8}}
\put(20,2){\line(1,0){8}}
\put(20,3){\line(1,0){8}}
\put(20,4){\line(1,0){8}}
\put(20,5){\line(1,0){8}}
\put(20,6){\line(1,0){8}}
\put(20,7){\line(1,0){8}}
\put(20,8){\line(1,0){8}}
\put(20,0){\line(0,1){8}}
\put(21,0){\line(0,1){8}}
\put(22,0){\line(0,1){8}}
\put(23,0){\line(0,1){8}}
\put(24,0){\line(0,1){8}}
\put(25,0){\line(0,1){8}}
\put(26,0){\line(0,1){8}}
\put(27,0){\line(0,1){8}}
\put(28,0){\line(0,1){8}}
\end{picture}
\end{center}
\caption{Uniform mesh\label{fig:MiSP4-RegularMesh}}
\end{figure}

\begin{table}[!h]\renewcommand{\baselinestretch}{1.10}\small
\begin{center}
 \caption{Results of error on uniform mesh with MiSP4\label{table:MiSP4-SquarePlate-ConvergenceOrder}}
\begin{tabular}{cccccccc}\hline
$t$ &                         &$4\times 4$ &$8\times 8$ &$16\times 16$ &$32\times 32$ &$64\times 64$ &rate\\\hline
1   &$|w-w_h|_1$                  &0.2806    &0.1460    &0.0736    &0.0369    &0.0184    &0.9819\\
    &$|\beta-\beta_h|_1$          &0.0771    &0.0383    &0.0191    &0.0095    &0.0048    &1.0039\\
    &$\|M-M_h\|_0$                &0.0062    &0.0020    &0.0008    &0.0003    &0.0002    &1.2977\\
    &$\|\gamma-\gamma_h\|_0$      &0.0877    &0.0458    &0.0231    &0.0116    &0.0058    &0.9799\\
    &$(t+h)\|\gamma-\gamma_h\|_0$ &0.1187    &0.0539    &0.0252    &0.0121    &0.0059    &1.0812\\\hline
0.1 &$|w-w_h|_1$                  &0.0117    &0.0052    &0.0025    &0.0012    &0.0006    &1.0610\\
    &$|\beta-\beta_h|_1$          &0.0775    &0.0384    &0.0191    &0.0095    &0.0048    &1.0057\\
    &$\|M-M_h\|_0$                &0.0061    &0.0020    &0.0008    &0.0003    &0.0002    &1.2957\\
    &$\|\gamma-\gamma_h\|_0$      &0.0870    &0.0458    &0.0231    &0.0116    &0.0058    &0.9771\\
    &$(t+h)\|\gamma-\gamma_h\|_0$ &0.0395    &0.0127    &0.0044    &0.0017    &0.0007    &1.4504\\\hline
0.001 &$|w-w_h|_1$                &0.0095    &0.0041    &0.0019    &0.0009    &0.0005    &1.0896\\
    &$|\beta-\beta_h|_1$          &0.0777    &0.0384    &0.0191    &0.0095    &0.0048    &1.0065\\
    &$\|M-M_h\|_0$                &0.0061    &0.0020    &0.0008    &0.0003    &0.0002    &1.2944\\
    &$\|\gamma-\gamma_h\|_0$      &0.0866    &0.0460    &0.0234    &0.0117    &0.0059    &0.9704\\
    &$(t+h)\|\gamma-\gamma_h\|_0$ &0.0307    &0.0082    &0.0021    &0.0005    &0.0001    &1.9555\\\hline
1e-8 &$|w-w_h|_1$                 &0.0095    &0.0041    &0.0019    &0.0009    &0.0005    &1.0896\\
    &$|\beta-\beta_h|_1$          &0.0777    &0.0384    &0.0191    &0.0095    &0.0048    &1.0065\\
    &$\|M-M_h\|_0$                &0.0061    &0.0020    &0.0008    &0.0003    &0.0002    &1.2944\\
    &$\|\gamma-\gamma_h\|_0$      &0.0866    &0.0460    &0.0234    &0.0117    &0.0059    &0.9703\\
    &$(t+h)\|\gamma-\gamma_h\|_0$ &0.0306    &0.0081    &0.0021    &0.0005    &0.0001    &1.9703\\\hline
\end{tabular}
\end{center}
\end{table}

 We note that the error analysis for MiSP4 element requires the partitions of domain to satisfy Assumption \ref{assumption}. However,  numerical results  in Table \ref{table:MiSP4-SquarePlate-NonConvergenceOrder} show that this assumption seems  not to be absolutely necessary for the uniform convergence, as is similar to the MITC4 element  \cite{Duran-Hernandez-Nieto-Liberman-Rodriguez2004}. Here the used partitions (Figure \ref{fig:MiSP4-IrregularMesh})  do not satisfy Assumption \ref{assumption}.

 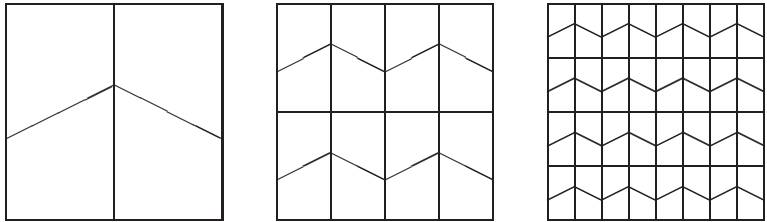
\begin{figure}[!h]
\begin{center}
\setlength{\unitlength}{0.36cm}
\begin{picture}(28,9)
\put(0,0){\line(1,0){8}}
\put(0,3){\line(2,1){4}}
\put(4,5){\line(2,-1){4}}
\put(0,8){\line(1,0){8}}
\put(0,0){\line(0,1){8}}
\put(4,0){\line(0,1){8}}
\put(8,0){\line(0,1){8}}

\put(10,0){\line(1,0){8}}
\put(10,1.5){\line(2,1){2}} \put(14,1.5){\line(2,1){2}}
\put(12,2.5){\line(2,-1){2}} \put(16,2.5){\line(2,-1){2}}
\put(10,4){\line(1,0){8}}
\put(10,5.5){\line(2,1){2}} \put(14,5.5){\line(2,1){2}}
\put(12,6.5){\line(2,-1){2}} \put(16,6.5){\line(2,-1){2}}
\put(10,8){\line(1,0){8}}
\put(10,0){\line(0,1){8}}
\put(12,0){\line(0,1){8}}
\put(14,0){\line(0,1){8}}
\put(16,0){\line(0,1){8}}
\put(18,0){\line(0,1){8}}

\put(20,0){\line(1,0){8}}
\put(20,0.75){\line(2,1){1}} \put(22,0.75){\line(2,1){1}}   \put(24,0.75){\line(2,1){1}} \put(26,0.75){\line(2,1){1}}
\put(21,1.25){\line(2,-1){1}} \put(23,1.25){\line(2,-1){1}}\put(25,1.25){\line(2,-1){1}} \put(27,1.25){\line(2,-1){1}}
\put(20,2){\line(1,0){8}}
\put(20,2.75){\line(2,1){1}} \put(22,2.75){\line(2,1){1}}   \put(24,2.75){\line(2,1){1}} \put(26,2.75){\line(2,1){1}}
\put(21,3.25){\line(2,-1){1}} \put(23,3.25){\line(2,-1){1}}\put(25,3.25){\line(2,-1){1}} \put(27,3.25){\line(2,-1){1}}
\put(20,4){\line(1,0){8}}
\put(20,4.75){\line(2,1){1}} \put(22,4.75){\line(2,1){1}}   \put(24,4.75){\line(2,1){1}} \put(26,4.75){\line(2,1){1}}
\put(21,5.25){\line(2,-1){1}} \put(23,5.25){\line(2,-1){1}}\put(25,5.25){\line(2,-1){1}} \put(27,5.25){\line(2,-1){1}}
\put(20,6){\line(1,0){8}}
\put(20,6.75){\line(2,1){1}} \put(22,6.75){\line(2,1){1}}   \put(24,6.75){\line(2,1){1}} \put(26,6.75){\line(2,1){1}}
\put(21,7.25){\line(2,-1){1}} \put(23,7.25){\line(2,-1){1}}\put(25,7.25){\line(2,-1){1}} \put(27,7.25){\line(2,-1){1}}
\put(20,8){\line(1,0){8}}
\put(20,0){\line(0,1){8}}
\put(21,0){\line(0,1){8}}
\put(22,0){\line(0,1){8}}
\put(23,0){\line(0,1){8}}
\put(24,0){\line(0,1){8}}
\put(25,0){\line(0,1){8}}
\put(26,0){\line(0,1){8}}
\put(27,0){\line(0,1){8}}
\put(28,0){\line(0,1){8}}
\end{picture}
\end{center}
\caption{Quadrilateral mesh\label{fig:MiSP4-IrregularMesh}}
\end{figure}

\begin{table}[!h]\renewcommand{\baselinestretch}{1.10}\small
\begin{center}
 \caption{Results of error on quadrilateral mesh with MiSP4\label{table:MiSP4-SquarePlate-NonConvergenceOrder}}
\begin{tabular}{cccccccc}\hline
$t$ &    &$4\times 4$ &$8\times 8$ &$16\times 16$ &$32\times 32$ &$64\times 64$ &rate\\\hline
1   &$|w-w_h|_1$                  &0.2873    &0.1693    &0.0881    &0.0445    &0.0223    &0.9217\\
    &$|\beta-\beta_h|_1$          &0.0924    &0.0528    &0.0255    &0.0122    &0.0060    &0.9872\\
    &$\|M-M_h\|_0$                &0.0066    &0.0032    &0.0012    &0.0005    &0.0002    &1.1968\\
    &$\|\gamma-\gamma_h\|_0$      &0.0899    &0.0531    &0.0277    &0.0140    &0.0070    &0.9203\\
    &$(t+h)\|\gamma-\gamma_h\|_0$ &0.1285    &0.0645    &0.0306    &0.0147    &0.0072    &1.0398\\\hline
0.1 &$|w-w_h|_1$                  &0.0118    &0.0064    &0.0031    &0.0015    &0.0008    &0.9898\\
    &$|\beta-\beta_h|_1$          &0.0834    &0.0496    &0.0253    &0.0122    &0.0060    &0.9506\\
    &$\|M-M_h\|_0$                &0.0065    &0.0031    &0.0012    &0.0005    &0.0002    &1.1925\\
    &$\|\gamma-\gamma_h\|_0$      &0.0930    &0.0574    &0.0285    &0.0141    &0.0070    &0.9318\\
    &$(t+h)\|\gamma-\gamma_h\|_0$ &0.0493    &0.0181    &0.0059    &0.0022    &0.0009    &1.4475\\\hline
0.001 &$|w-w_h|_1$                &0.0096    &0.0051    &0.0024    &0.0012    &0.0006    &1.0151\\
    &$|\beta-\beta_h|_1$          &0.0835    &0.0475    &0.0245    &0.0120    &0.0060    &0.9525\\
    &$\|M-M_h\|_0$                &0.0066    &0.0032    &0.0013    &0.0006    &0.0003    &1.1477\\
    &$\|\gamma-\gamma_h\|_0$      &0.0947    &0.0702    &0.0470    &0.0355    &0.0310    &0.4031\\
    &$(t+h)\|\gamma-\gamma_h\|_0$ &0.0408    &0.0152    &0.0051    &0.0019    &0.0009    &1.3908\\\hline
1e-8  &$|w-w_h|_1$                &0.0096    &0.0051    &0.0024    &0.0012    &0.0006    &1.0151\\
    &$|\beta-\beta_h|_1$          &0.0835    &0.0475    &0.0245    &0.0120    &0.0060    &0.9525\\
    &$\|M-M_h\|_0$                &0.0066    &0.0032    &0.0013    &0.0006    &0.0003    &1.1466\\
    &$\|\gamma-\gamma_h\|_0$      &0.0947    &0.0703    &0.0470    &0.0356    &0.0315    &0.3976\\
    &$(t+h)\|\gamma-\gamma_h\|_0$ &0.0407    &0.0151    &0.0051    &0.0019    &0.0008    &1.3976\\\hline
\end{tabular}
\end{center}
\end{table}
\ \\

{\bf Acknowledgements.}
The work of the first author was partly supported by National Natural Science Foundation of China (11401492 and 11226333).
The work of the second author was partly supported by National Natural Science Foundation of China (11171239) and Major Research Plan of  National Natural Science Foundation of China (91430105).


\begin{thebibliography}{10}

\bibitem{Arnold.D1981}
D.N. Arnold.
\newblock Discretization by finite elements of a model parameter dependent
  problem.
\newblock {\em Numerische Mathematik}, 37(3):405--421, 1981.

\bibitem{Arnold.D;Falk.R1989}
D.N. Arnold and R.S. Falk.
\newblock A uniformly accurate finite element method for the {Reissner-Mindlin}
  plate.
\newblock {\em SIAM Journal on Numerical Analysis}, 26(6):1276--1290, 1989.

\bibitem{Ayad.R;Dhatt.G;Batoz.J1998}
R.~Ayad, G.~Dhatt, and J.L. Batoz.
\newblock {A new hybrid-mixed variational approach for Reissner--Mindlin
  plates. The MiSP model}.
\newblock {\em International journal for numerical methods in engineering},
  42(7):1149--1179, 1998.

\bibitem{bathe1989mitc7}
K.J. Bathe, F.~Brezzi, and S.W. Cho.
\newblock The mitc7 and mitc9 plate bending elements.
\newblock {\em Computers \& Structures}, 32(3):797--814, 1989.

\bibitem{Bathe.K;Dvorkin.E1985}
K.J. Bathe and E.N. Dvorkin.
\newblock A four-node plate bending element based on {Mindlin/Reissner} plate
  theory and a mixed interpolation.
\newblock {\em International Journal for Numerical Methods in Engineering},
  21(2):367--383, 1985.

\bibitem{batoz1980study}
J.L. Batoz, K.J. Bathe, and L.W. Ho.
\newblock A study of three-node triangular plate bending elements.
\newblock {\em International Journal for Numerical Methods in Engineering},
  15(12):1771--1812, 1980.

\bibitem{Batoz.J;Bentahar.M1982}
J.L. Batoz and M.B. Tahar.
\newblock Evaluation of a new quadrilateral thin plate bending element.
\newblock {\em International Journal for Numerical Methods in Engineering},
  18(11):1655--1677, 1982.

\bibitem{Boffi2008}
Daniele Boffi and Lucia Gastaldi.
\newblock {\em Mixed finite elements, compatibility conditions, and
  applications: lectures given at the CIME Summer School held in Cetraro,
  Italy, June 26-July 1, 2006}, volume 1939.
\newblock Springer, 2008.

\bibitem{Braess.D2001}
D.~Braess.
\newblock {\em Finite elements: Theory, fast solvers, and applications in solid
  mechanics}.
\newblock Cambridge Univ Pr, 2001.

\bibitem{Brezzi.F;Bathe.KJ;Fortin.M1989}
F.~Brezzi, K.J. Bathe, and M.~Fortin.
\newblock Mixed-interpolated elements for {Reissner--Mindlin} plates.
\newblock {\em International Journal for Numerical Methods in Engineering},
  28(8):1787--1801, 1989.

\bibitem{Brezzi.F;Fortin.M1991}
F.~Brezzi and M.~Fortin.
\newblock {\em Mixed and hybrid finite element methods}.
\newblock Springer-Verlag, 1991.

\bibitem{Brezzi.F;Fortin.M;Stenberg.R1991}
F.~Brezzi, M.~Fortin, and R.~Stenberg.
\newblock Error analysis of mixed-interpolated elements for {Reissner-Mindlin}
  plates.
\newblock {\em Math. Models Methods Appl. Sci}, 1(2):125--151, 1991.

\bibitem{C;Xie;Yu;Zhou2011}
C.~Carstensen, X.~Xie, G.~Yu, and T.~Zhou.
\newblock A priori and a posteriori analysis for a locking-free low order
  quadrilateral hybrid finite element for {Reissner-Mindlin} plates.
\newblock {\em Computer Methods in Applied Mechanics and Engineering},
  200(9-12):1161--1175, 2011.

\bibitem{xiaoliang1994simple}
X.L. Cheng.
\newblock A simple finite element method for the reissner-mindlin plate ").
\newblock {\em J. Comput. Math}, 12(1):46--54, 1994.

\bibitem{DURAN.R;LIBERMAN.E1992}
R.~Dur{\'a}n and E.~Liberman.
\newblock On mixed finite element methods for the {Reissner-Mindlin} plate
  model.
\newblock {\em Mathematics of computation}, 58(198):561--573, 1992.

\bibitem{Duran-Hernandez-Nieto-Liberman-Rodriguez2004}
R.G. Dur{\'a}n, E.~Hern{\'a}ndez, L.~Hervella-Nieto, E.~Liberman, and
  R.~Rodr{\'\i}guez.
\newblock Error estimates for low-order isoparametric quadrilateral finite
  elements for plates.
\newblock {\em SIAM journal on numerical analysis}, 41:1751--1772, 2003.

\bibitem{FALK.R;TU.T2000}
R.S. Falk and T.~Tu.
\newblock Locking-free finite elements for the {Reissner-Mindlin} plate.
\newblock {\em Mathematics of computation}, 69(231):911--928, 2000.

\bibitem{Girault-Raviart}
V.~Girault and P.A. Raviart.
\newblock {Finite element methods for Navier-Stokes equations, Theory and
  algorithms, volume 5 of Springer Series in Computational Mathematics}, 1986.

\bibitem{Hu.J;Ming;Shi2003}
J.~Hu, P.~Ming, and Z.~Shi.
\newblock Nonconforming quadrilateral rotated $q^1$ element for
  {Reissner-Mindlin} plate.
\newblock {\em Journal of Computational Mathematics}, 21(1):25--32, 2003.

\bibitem{Hu.J;Shi.Z2007}
J.~Hu and Z.C. Shi.
\newblock Two lower order nonconforming rectangular elements for the
  {Reissner-Mindlin} plate.
\newblock {\em Mathematics of computation}, 76(260):1771--1786, 2007.

\bibitem{Hu.J;Shi.Z2008}
J.~Hu and Z.C. Shi.
\newblock Error analysis of quadrilateral wilson element for
  {Reissner--Mindlin} plate.
\newblock {\em Computer Methods in Applied Mechanics and Engineering},
  197(6):464--475, 2008.

\bibitem{Hu.J;Shi.Z2009}
J.~Hu and Z.C. Shi.
\newblock {Analysis for quadrilateral MITC elements for the Reissner-Mindlin
  plate problem}.
\newblock {\em Mathematics of computation}, 78(266):673--711, 2009.

\bibitem{Hughes.T1987}
T.J.R. Hughes.
\newblock {\em The finite element method: linear static and dynamic finite
  element analysis}.
\newblock Prentice-hall, 1987.

\bibitem{Hughes;1978}
T.J.R. Hughes, M.~Cohen, and M.~Haroun.
\newblock Reduced and selective integration techniques in the finite element
  analysis of plates.
\newblock {\em Nuclear Engineering and Design}, 46(1):203--222, 1978.

\bibitem{hughes1981linear}
T.J.R. Hughes and R.L. Taylor.
\newblock The linear triangular bending element.
\newblock {\em The Mathematics of Finite Elements and Applications},
  4:127--142, 1981.

\bibitem{Hughes;1977}
T.J.R. Hughes, R.L. Taylor, and W.~Kanoknukulchai.
\newblock A simple and efficient finite element for plate bending.
\newblock {\em International Journal for Numerical Methods in Engineering},
  11(10):1529--1543, 1977.

\bibitem{hughes1981finite}
T.J.R. Hughes and T.E. Tezduyar.
\newblock {Finite elements based upon Mindlin plate theory with particular
  reference to the four-node bilinear isoparametric element}.
\newblock {\em Journal of Applied Mechanics}, 48:587, 1981.

\bibitem{LOVADINA.C2005}
C.~Lovadina.
\newblock {A low-order nonconforming finite element for Reissner-Mindlin
  plates}.
\newblock {\em SIAM journal on numerical analysis}, 42(6):2688--2705, 2005.

\bibitem{macneal1982derivation}
R.H. Macneal.
\newblock Derivation of element stiffness matrices by assumed strain
  distributions.
\newblock {\em Nuclear Engineering and Design}, 70(1):3--12, 1982.

\bibitem{Malkus;1978}
D.S. Malkus and T.J.R. Hughes.
\newblock Mixed finite element methods--reduced and selective integration
  techniques:{ A }unification of concepts.
\newblock {\em Computer Methods in Applied Mechanics and Engineering},
  15(1):63--81, 1978.

\bibitem{MING.P;SHI.Z2001}
P.B. Ming and Z.C. Shi.
\newblock Nonconforming rotated $q^1$ element for {Reissner-Mindlin} plate.
\newblock {\em Mathematical Models and Methods in Applied Sciences},
  11(8):1311--1342, 2001.

\bibitem{MING.P;SHI.Z2005}
P.B. Ming and Z.C. Shi.
\newblock {Two nonconforming quadrilateral elements for the Reissner-Mindlin
  plate}.
\newblock {\em Mathematical Models and Methods in Applied Sciences},
  15(10):1503--1518, 2005.

\bibitem{Ming.P;Shi.Z2006}
P.B. Ming and Z.C. Shi.
\newblock Analysis of some low order quadrilateral reissner-mindlin plate
  elements.
\newblock {\em Mathematics of computation}, 75(255):1043--1065, 2006.

\bibitem{papadopoulos1990triangular}
P.~Papadopoulos and R.L. Taylor.
\newblock {A triangular element based on Reissner-Mindlin plate theory}.
\newblock {\em International journal for numerical methods in engineering},
  30(5):1029--1049, 1990.

\bibitem{Pitkaranta.J;Suri.M1996}
J.~Pitk{\"a}ranta and M.~Suri.
\newblock Design principles and error analysis for reduced-shear plate-bending
  finite elements.
\newblock {\em Numerische Mathematik}, 75(2):223--266, 1996.

\bibitem{Zhang.Z;Zhang.S1994}
Z.~Zhang and S.~Zhang.
\newblock {Wilson's element for the Reissner-Mindlin plate}.
\newblock {\em Computer methods in applied mechanics and engineering},
  113(1-2):55--65, 1994.

\bibitem{zienkiewicz1990plate}
O.C. Zienkiewicz, R.L. Taylor, P.~Papadopoulos, and E.~Onate.
\newblock Plate bending elements with discrete constraints: new triangular
  elements.
\newblock {\em Computers \& Structures}, 35(4):505--522, 1990.

\bibitem{Zienkiewicz;1971}
O.C. Zienkiewicz, R.L. Taylor, and J.M. Too.
\newblock Reduced integration technique in general analysis of plates and
  shells.
\newblock {\em International Journal for Numerical Methods in Engineering},
  3(2):275--290, 1971.

\end{thebibliography}
\end{document}